\documentclass[10pt,reqno]{amsart}
\usepackage{verbatim}
\usepackage{amsmath}
\usepackage{amssymb, amsfonts, euscript, enumerate,latexsym}
\usepackage{amsthm} 
\usepackage{pxfonts}
\usepackage{fancyhdr}
\usepackage{mathrsfs}
\usepackage{mathtools}
\usepackage{epsfig,subfigure}
\usepackage{graphicx}
\usepackage{color,wrapfig}
\usepackage{txfonts}
\usepackage{hyperref}

\newtheorem{theorem}{Theorem}[section]
\newtheorem{lem}[theorem]{Lemma}
\newtheorem{cor}[theorem]{Corollary}
\newtheorem{remark}[theorem]{Remark}
\newtheorem{prop}[theorem]{Proposition}

\theoremstyle{definition}

\newcommand{\re}{{\mathbb R}} 
\newcommand{\R}{{\mathbb R}}

\newcommand{\Z}{{\mathbb Z}}

\newcommand{\ld}{\lambda}

\newcommand{\ve}{\varepsilon}

\newcommand{\loc}{\operatorname{loc}}
\newcommand{\D}{\nabla}

\newcommand{\La}{\Delta}

\newcommand{\2}{\overline}
\newcommand{\3}{\varepsilon}
\newcommand{\4}{\widetilde}

\begin{document}

\title [Asymptotic behaviour of singular solution]{Asymptotic behaviour of singular solution of the fast diffusion equation in the punctured Euclidean space}

\author{Kin Ming Hui}
\address{Institute of Mathematics, 
Academia Sinica,
Taipei, Taiwan, R. O. C.}
\email{kmhui@gate.sinica.edu.tw}

\author{Jinwan Park}
\address{Research Institute of Mathematics,
 Seoul National University,
Seoul 08826, Korea}
\email{jinwann@snu.ac.kr}

\date{\today}
\thanks{
}


\begin{abstract}
For $n\ge 3$, $0<m<\frac{n-2}{n}$, $\beta<0$ and $\alpha=\frac{2\beta}{1-m}$, we prove the existence, uniqueness and asymptotics near the origin of the singular eternal self-similar solutions of the fast diffusion equation in $(\mathbb{R}^n\setminus\{0\})\times \mathbb{R}$ 
of the form $U_{\lambda}(x,t)=e^{-\alpha t}f_{\lambda}(e^{-\beta t}x), x\in \mathbb{R}^n\setminus\{0\}, t\in\mathbb{R},$ where $f_{\lambda}$ is a radially symmetric function satisfying 
$$\frac{n-1}{m}\Delta f^m+\alpha f+\beta x\cdot\nabla f=0 \text{ in }\mathbb{R}^n\setminus\{0\},$$
with $\underset{\substack{r\to 0}}{\lim}\frac{r^2f(r)^{1-m}}{\log r^{-1}}=\frac{2(n-1)(n-2-nm)}{|\beta|(1-m)}$ and $\underset{\substack{r\to\infty}}{\lim}r^{\frac{n-2}{m}}f(r)=\lambda^{\frac{2}{1-m}-\frac{n-2}{m}}$, for some constant $\lambda>0$.

As a consequence we prove the existence and uniqueness of solutions of Cauchy problem for the fast diffusion equation $u_t=\frac{n-1}{m}\Delta u^m$ in $(\mathbb{R}^n\setminus\{0\})\times (0,\infty)$ with initial value $u_0$ satisfying $f_{\lambda_1}(x)\le u_0(x)\le f_{\lambda_2}(x)$, $\forall x\in\mathbb{R}^n\setminus\{0\}$, which satisfies $U_{\lambda_1}(x,t)\le u(x,t)\le U_{\lambda_2}(x,t)$, $\forall x\in \mathbb{R}^n\setminus\{0\}, t\ge 0$, for some constants $\lambda_1>\lambda_2>0$. 

We also prove the asymptotic behaviour of such singular solution $u$ of the fast diffusion equation as $t\to\infty$ when  $n=3,4$ and $\frac{n-2}{n+2}\le m<\frac{n-2}{n}$ holds. Asymptotic behaviour of such singular solution $u$ of the fast diffusion equation as $t\to\infty$  is also obtained when $3\le n<8$, $1-\sqrt{2/n}\le m<\min\left(\frac{2(n-2)}{3n},\frac{n-2}{n+2}\right)$,  and $u(x,t)$ is radially symmetric in $x\in\mathbb{R}^n\setminus\{0\}$ for any $t>0$ under appropriate conditions on the initial value $u_0$.
\end{abstract}

\maketitle 


\section{Introduction}
\setcounter{equation}{0}
\setcounter{theorem}{0}

Recently there is a lot of study on the equation 
\begin{equation}\label{fde}
u_t=\frac{n-1}{m} \Delta u^m, \quad u>0,
\end{equation}
in $\R^n\times (0,T)$, $T>0$, by D.G. Aronson \cite{A}, P. Daskalopoulos, J. King, M. del Pino, N. Sesum, M. S\'aez, \cite{DKS}, \cite{DPS}, \cite{DS1}, \cite{DS2}, \cite{PS}, S.Y. Hsu \cite{Hs1}, \cite{Hs2}, \cite{Hs3}, K.M. Hui \cite{H1}, \cite{H2}, \cite{H3}, M. Fila, M. Winkler, E. Yanagida, J.L. Vazquez \cite{FVWY}, \cite{FW1}, \cite{FW2}, \cite{FW3}, \cite{VW}, \cite{V1}, etc. We refer the readers to the survey paper \cite{A}
and the books \cite{DK}, \cite{V2} on the recent results of \eqref{fde}.

For $m>1$, \eqref{fde} arises in the flow of gases through porous media or oil passing through sand, etc., and it is called the porous medium equation. For $m=1$, \eqref{fde} is the heat equation. For $0 <m<1$, \eqref{fde} is called the fast diffusion equation. If $g=u^{\frac{4}{n+2}}dx^2$ is a metric on $\re^n$, $n \ge 3$, then $g$ satisfies the Yamabe flow, 
\begin{equation*}
\frac{\partial g}{\partial t} =-Rg\quad\mbox{ in } \R^n\times (0,T),
\end{equation*}
if and only if $u$ satisfies \eqref{fde} in $\R^n\times (0,T)$ with $m=\frac{n-2}{n+2}$. 

As observed by J.L.~Vazquez \cite{V1} and others there is a considerable difference in the behaviour of the solutions of \eqref{fde} for the cases $0<m<\frac{(n-2)_+}{n}$, $\frac{(n-2)_+}{n}<m<1$, and $m>1$. For example when $m>1$, if the initial value $0\le u_0\in L^1(\R^n)$  has compact support, then the solution $u$ of  
\begin{align}\label{Cauchy-problem}
\begin{cases}
u_t=\frac{n-1}{m}\Delta u^m& \quad \text{ in } \re^n  \times (0, T)\\
u(\cdot,0)=u_0&\quad \text{ in } \re^n
\end{cases}
\end{align}
will have compact support for any $0<t<T$ \cite{A}. On the other hand when $0<m<1$ and $0\le u_0\in L^1(\R^n)$  has compact support, M.A.~Herrero and M.~Pierre \cite{HP} proved that the solution $u$ of \eqref{Cauchy-problem} is positive on $\R^n\times (0,T)$. M.A.~Herrero and M.~Pierre \cite{HP} also proved that when $\frac{n-2}{n}<m<1$ and $0\le u_0\not\equiv 0$, there exists a unique global positive solution of \eqref{Cauchy-problem}. On the other hand when $0<m<\frac{(n-2)_+}{n}$, it is known that there exist solutions of \eqref{Cauchy-problem} which extinct in a finite time. For example the Barenblatt solution \cite{DS1},
\begin{equation*}
B_k(x,t)=(T-t)^{\frac{n}{n-2-nm}}\left(\frac{C^{\ast}}{k+(T-t)^{\frac{2}{n-2-nm}}|x|^2}\right)^{\frac{1}{1-m}},\quad k>0,
\end{equation*}
where 
\begin{equation*}
C^{\ast}=\frac{2(n-1)(n-2-nm)}{1-m}
\end{equation*}
is a positive solution of \eqref{Cauchy-problem} with $u_0=B_k(x,0)$ which vanishes identically at $t=T$.
 
There is a lot of research on \eqref{fde} for the case $0<m<\frac{n-2}{n}$ and $n\ge 3$ recently  by P. Daskalopoulos, J. King, M. del Pino, N. Sesum \cite{DKS}, \cite{DPS}, S.Y. Hsu \cite{Hs1}, \cite{Hs2}, \cite{Hs3}, K.M. Hui \cite{H1}, \cite{H2}, \cite{H3}, J.L. Vazquez \cite{V1}, etc. On the other hand various singular solutions of \eqref{fde} in the Euclidean space minus a finite number of points which blow up either at the origin or at a finite number of points for the case $0<m<\frac{n-2}{n}$, $n\ge 3$, were studied by K.M. Hui, Soojung Kim, Sunghoon Kim, T.~Jin and J.~Xiong, etc. \cite{HK}, \cite{HKs}, \cite{JX}. 

In this paper, for $0<m<\frac{n-2}{n}$ and $n\ge 3$, we study the existence and uniqueness of eternal self-similar solutions of \eqref{fde} which blow up at the origin for all time $t\in\R$ with specific growth rate at $0$. Moreover, we study the existence, uniqueness, and asymptotic large time behaviour of the singular solution to the Cauchy problem,
\begin{align}\label{main-eq}
\begin{cases}
u_t=\frac{n-1}{m}\Delta u^m& \quad \text{ in } (\re^n \setminus \{0\}) \times (0, \infty)\\
u(\cdot,0)=u_0&\quad \text{ in } \re^n \setminus \{0\},
\end{cases}
\end{align}
which lying between two singular self-similar solutions.

The main difficulty of the theory is to find appropriate \emph{weighted $L^1$-contraction}, the weighted $L^1$-norm of the difference of two solutions is nonincreasing with respect to the time variable. The property is mainly used in the proof of the uniqueness and asymptotic large time behaviour of the solution of \eqref{main-eq}.

In the study of the $L^1$-contraction, we found out that the growth rate of solutions at the origin is relatively high to have appropriate $L^1$-contraction for the uniqueness and large time asymptotic behaviour, if we take $|x|^{-\mu}$ ($0<\mu\le n- 2/(1-m)$) by the weight function and estimate the difference of two solutions by a self-similar solution as in \cite{HK}.

Thus, in Section \ref{section-Higher order asymptotics}, for the first step, we suggest a more accurate estimate on the difference of two self-similar solutions, Corollary \ref{dif f}, by using the higher order asymptocity of self-similar solutions, Theorem \ref{Thm HAB}. Then, we have $L^1$-contraction with weight function $|x|^{-\mu}$, Theorem \ref{u-v-weight-mu-L1-contraction} and have the uniqueness of the solution of \eqref{main-eq} lying between two singular self-similar solutions, Theorem \ref{thm exi uni u}.

For the asymptotic behaviour of the solution, we need more strong property on the weighted $L^1$-contraction that the weighted $L^1$-norm of the difference of two rescaled solutions vanishes as time goes to infinity. However, the property is not obtained by the $L^1$-contraction with weight function $|x|^{-\mu}$.

Therefore, for $n=3, 4$, $\frac{n-2}{n+2}\le m < \frac{n-2}{n}$, we introduce the $L^1$-contraction with a power of self-similar solution as the weight function, which implies the vanishing property and the asymptotic large time behaviour of the solution, Theorem \ref{u-v-L1-contraction2} and Theorem \ref{thm assm}. Furthermore, if the initial value of the solution is radially symmetric, then we have the asymptotic large time behaviour of the solution in higher dimension $3 \le n < 8$, $1- \sqrt{\frac{2}{n}} \le m \le \min \left \{\frac{2(n-2)}{3n}, \frac{n-2}{n+2}\right \}$, Theorem \ref{thm assm rad sym}.

\subsection{Contents and methodology}

In section \ref{section-eternal-self-similar-soln}, for $0<m<\frac{n-2}{n}$ and $n\ge 3$, we will study the existence and uniqueness of radially symmetric eternal self-similar solutions of \eqref{fde} in $(\R^n \setminus \{0\})\times \re$ of the form
\begin{equation*}
U(x,t):=e^{-\alpha t}f(e^{-\beta t}x)\quad \forall x\in \R^n \setminus \{0\},t\in \R
\end{equation*}
that blow up at $\{0\}\times\R$, where $f$ is a radially symmetric function satisfying
\begin{equation}\label{f-eqn}
\frac{n-1}{m}\Delta f^m+\alpha f+\beta x\cdot\nabla f=0, \quad f>0,\quad\mbox{ in }\R^n\setminus\{0\},
\end{equation}
or equivalently,
\begin{equation}\label{ode}
\frac{n-1}{m}\left((f^m)_{rr}+\frac{n-1}{r}(f^m)_r\right)+\alpha f+\beta rf_r=0,\quad f>0,\quad \forall r>0
\end{equation}
with
\begin{equation}\label{ab}
\beta<0 \quad \text{ and } \quad \alpha= \frac{2\beta}{1-m}
\end{equation}
and $f$  blows up at the origin, Theorem \ref{f-existence-uniqueness-thm}. Furthermore, we prove that such function $f$ satisfies
\begin{equation}\label{gro 0}
\lim_{r\to 0} \frac{r^{2}f(r)^{1-m}}{\log r^{-1}}=\frac{2(n-1)(n-2-nm)}{ (1-m)|\beta|}
\end{equation}
and
\begin{equation}\label{gro inf}
\lim_{r \to \infty} r^{\frac{n-2}{m}} f(r)=A 
\end{equation}
for some constant $A>0$, Theorem \ref{Thm HAB}.

We note that for $0<m<\frac{n-2}{n}$ and $n\ge 3$, forward self-similar solutions of \eqref{fde} which blow up at the origin for  $t>0$ are constructed by K.M.~Hui  and Soojung Kim  in \cite{HK}. The self-similar solutions constructed in \cite{HK} are of the form,
\begin{equation*}
U(x,t)=t^{-\alpha}f(t^{-\beta}x)\quad\forall 0\ne x\in\R^n, t>0,
\end{equation*}
where $f$ is a radially symmetric solution of \eqref{f-eqn} with $\alpha$, $\beta$ satisfying
\begin{equation}\label{ab-relation2}
\alpha=\frac{2\beta-1}{1-m},\quad \beta=\frac{1}{2-\gamma(1-m)}, \quad\mbox{ and }\quad
\frac{2}{1-m}<\gamma<\frac{n-2}{m}.
\end{equation}
It is proved in \cite{HK} that such function $f$ with $\alpha$, $\beta$ satisfying \eqref{ab-relation2} satisfies
\begin{equation*}
\lim_{r\to 0}r^{\alpha/\beta}f(r)=A\quad\mbox{ and }\quad
\lim_{r \to \infty} r^{\frac{n-2}{m}} f(r)=D_A 
\end{equation*}
for some constant $A>0$ and some constant $D_A>0$ depending on $A$. Thus the behaviour of the solution $f$ of \eqref{ode} with $\alpha$, $\beta$ satisfying \eqref{ab} that blows up at the origin is completely different from the behaviour of the solution $f$ of \eqref{ode} with $\alpha$, $\beta$ satisfying \eqref{ab-relation2} that blows up at the origin. 

In section \ref{section-Higher order asymptotics}, we prove the higher order asymptotics of the eternal self-similar solutions near the origin, Theorem \ref{Thm HAB}, by using modifications of the proofs in \cite{Hs4}.

In section \ref{section-existence-uniqueness-asymptotic-behaviour}, as a consequence, we will prove the existence, uniqueness, and asymptotic large time behaviour of the singular solutions of \eqref{main-eq}.


Precisely, by using Theorem \ref{Thm HAB}, we have $L^1$-contraction with weight $|x|^{-\mu}$ ($0<\mu\le n- 2/(1-m)$), Theorem \ref{thm exi uni u}. Therefore, we prove the existence and uniqueness of the singular solution to the Cauchy problem \eqref{main-eq}, which blows up at the origin for all time $t>0$ when $0<m<\frac{n-2}{n}$, $n\ge 3$, and the initial value $u_0$ lies between two solutions of \eqref{ode} which blow up at the origin with $\alpha$, $\beta$ satisfying \eqref{ab}. 

Furthermore, for $n= 3, 4$,  $\frac{n-2}{n+2}\le m<\frac{n-2}{n}$ and $\alpha$, $\beta$ given by \eqref{ab}, 
we prove $L^1$-contraction with a postive power of the self-similar function $f_\lambda$ as the weight function, Theorem \ref{u-v-L1-contraction2}. Therefore, we have the asymptotic large time behaviour of the singular solution of \eqref{main-eq}, under some appropriate conditions on the initial value $u_0$, Theorem \ref{thm assm}. More precisely, if $u$ is the solution of \eqref{main-eq}, then under appropriate condition on the initial value $u_0$, we proved that the rescaled function
\begin{equation}\label{rescaled-soln-defn}
\4{u}(x,t)=e^{\alpha t}u(e^{\beta t}x, t)\quad\forall x\in\re^n \setminus \{0\}, t\ge 0
\end{equation}
converges uniformly on every compact subset of $\R^n\setminus\{0\}$ to  some radially symmetric function $f$ which satisfies \eqref{ode} and blows up at the origin as $t\to\infty$. Note that the rescaled function $\4{u}$ satisfies
\begin{align}\label{main-eq-resc}
\begin{cases}
\widetilde u_t=\frac{n-1}{m} \Delta\widetilde{u}^m + \alpha \widetilde u + \beta y \cdot \nabla \widetilde u 
& \quad \text{ in } (\re^n \setminus \{0\}) \times (0, \infty), \\
\widetilde{u}(\cdot,0)=u_0 & \quad \text{ in } \re^n \setminus \{0\}.
\end{cases}
\end{align}

In section \ref{section-asymptotic-behaviour-radially-symmetric-singular-soln}, we prove the asymptotic large time behaviour of the radially symmetric singular solutions of \eqref{main-eq}, Theorem \ref{thm assm rad sym}. If $u$ is a radially symmetric solution of \eqref{fde}, then the function
\begin{equation}\label{u-bar-defn}
\2{u}(x,t)=|x|^{-\frac{n-2}{m}}u(r^{-1},t)\quad\forall x\in\R^n\setminus\{0\}, r=|x|>0, t>0,
\end{equation}
satisfies
\begin{equation}\label{u-bar-eqn}
\2{u}_t=\frac{n-1}{m}|x|^{n+2-\frac{n-2}{m}}\Delta \2{u}^m\quad\mbox{ in }(\R^n\setminus\{0\})\times (0,\infty).
\end{equation}
Hence if $u$ is a radially symmetric solution of \eqref{main-eq}, then the function $\2{u}$ given by \eqref{u-bar-defn} satisfies
\begin{align}\label{Inversed eq}
\begin{cases}
\overline u_t=\frac{n-1}{m} |x|^{n+2-\frac{n-2}{m}} \Delta \overline u^m & \quad \text{ in } (\re^n \setminus \{0\}) \times (0, \infty)\\
\overline u(\cdot,0)= r^{-\frac{n-2}{m}} u_0(r^{-1})=:\overline u_0 (r) & \quad \text{ in } \re^n \setminus \{0\}.
\end{cases}
\end{align}
By studying the asymptotic large time behaviour of the radially symmetric solution $\2{u}$ of
\eqref{Inversed eq}, we obtain the asymptotic large time behaviour of the radially symmetric solution $u$ of \eqref{main-eq}, which blow up at the origin with initial value $u_0$ lying between two solutions of \eqref{ode} that blow up at the origin
for some constants $\alpha$, $\beta$ satisfying \eqref{ab} for the case $3\le n<8$ and $1-\sqrt{2/n}\le m<\min\left(\frac{2(n-2)}{3n},\frac{n-2}{n+2}\right)$.

We note that in section \ref{section-eternal-self-similar-soln}, the inversion technique of K.M.~Hui and Soojung Kim \cite{HK} is exploited to study the existence and uniqueness of the solution of \eqref{ode} with $\alpha$, $\beta$ satisfying \eqref{ab}. 
More precisely, if $f$ is a solution of \eqref{ode}, then as proved by K.M.~Hui and Soojung Kim \cite{HK} the function  $g$ given by 
\begin{equation}\label{def g}
g(r):=r^{-\frac{n-2}{m}}f(r^{-1})\quad\forall r>0,
\end{equation}
satisfies 
\begin{equation}\label{inv eq}
\frac{n-1}{m}\left((g^m)_{rr}+\frac{n-1}{r}(g^m)_r\right)+ r^{\frac{n-2}{m}-n-2}\left ( \widetilde \alpha g+\widetilde\beta rg_r\right )=0,\quad g>0,\quad\mbox{ in }(0,\infty),
\end{equation}
where
\begin{equation}\label{tilde ab}
\widetilde \alpha=\alpha-\frac{n-2}{m}\beta\quad \text{ and } \quad \widetilde \beta=-\beta.
\end{equation}
Conversely, if $g$ is a solution of \eqref{inv eq} with $\alpha$, $\beta$, $\4{\alpha}$, $\4{\beta}$ satisfying \eqref{ab} and \eqref{tilde ab} respectively, then the function $f$ given by 
\begin{equation}\label{f-g-relation}
f(r):=r^{-\frac{n-2}{m}}g(r^{-1})\quad\forall r>0,
\end{equation}
satisfies \eqref{ode} with $\alpha$, $\beta$ satisfying \eqref{ab}. Hence, the existence of a solution $f$ of \eqref{ode} which blows up at the origin follows from the existence of a solution $g$ of \eqref{inv eq} which was proved in \cite{HK}. We will also use a modification of the proof of S.Y.~Hsu \cite{Hs1} to prove the growth rate of the solution $g$ of \eqref{inv eq} at infinity which by \eqref{f-g-relation} then implies the blow up rate of $f$ at the origin. 

In the rest of the paper, unless otherwise stated, we will let $n\ge 3$, $0<m<\frac{n-2}{n}$ and let $\alpha$, $\beta$, $\4{\alpha}$, $\4{\beta}$ be as given by \eqref{ab} and  \eqref{tilde ab}, respectively. If $u$ is the solution of \eqref{fde}, we will let $\4{u}$ and $\2{u}$ be as given by \eqref{rescaled-soln-defn} and \eqref{u-bar-defn}, respectively. Note that by \eqref{ab} and \eqref{tilde ab}, 
\begin{equation}\label{alpha-beta-and-tilde-relation}
\frac{\widetilde \alpha}{\widetilde \beta}=-\frac{2}{1-m}+ \frac{n-2}{m}\in \left(0,\frac{n-2}{m}\right), \quad \widetilde \beta>0, \quad \text{ and } \quad \widetilde \alpha>0.
\end{equation}


\subsection{Definitions}
 
For any function $k : \re^n \setminus \{0\} \to [0,\infty)$, let
\begin{equation*}
L^1 \left (k; \R^n\setminus\{0\}\right ):=\left\{ h\in L_{loc}^1(\R^n\setminus\{0\}): \int_{\R^n\setminus\{0\}} |h(x)| k(x) \, dx<\infty\right\}
\end{equation*}
with the norm 
$$\|h\|_{L^1 \left ( k; \R^n\setminus\{0\}\right )}=\int_{\R^n\setminus\{0\}} |h(x)| k(x) \, dx.$$

 For any $0\leq u_0\in L^1_{\loc}(\R^n\setminus\{0\}),$ we say that $u$ is a solution of \eqref{main-eq} 
if $u>0$ in $(\R^n\setminus\{0\})\times(0, \infty)$ satisfies \eqref{fde} in $ \left(\R^n\setminus\{0\}\right)\times(0, \infty)$ in the classical sense and 
\begin{equation*}
\| u(\cdot, t)-u_0\|_{L^1(K)}\to 0\quad\mbox{ as }t\to 0
\end{equation*} 
for any compact set $K\subset \R^n\setminus\{0\}.$

 For any $x_0\in\R^n$ and $R>0$, we let $B_{R}(x_0)=\{x\in\R^n: |x-x_0|<R\}$ and $B_R=B_R(0)$. For any $R>1$, let $A_R=\{x\in\R^n:1/R<|x|<R\}$. For any set $E$ we let $\chi_E$ be the characteristic function of $E$. We also let $\omega_n$ denote the surface area of
 $S^{n-1}=\{x\in\R^n:|x|=1\}$. 

\subsection{Main theorems}

The main results are the following.

\begin{theorem}[Existence of self-similar profile] \label{f-existence-uniqueness-thm}
Let $n\ge 3$, $0<m<\frac{n-2}{n}$, and  $\alpha$, $\beta$ be as given by \eqref{ab}. Then for any constant $A>0$ there exists a unique solution $f=f_{\beta,A}$ of \eqref{ode} which satisfies  \eqref{gro inf}. Moreover, $f$ satifies \eqref{gro 0} and
\begin{equation}\label{f dec}
\alpha f(r)+\beta r f_r(r)>0 \quad \forall r>0.
\end{equation}
\end{theorem}

Let $f_1$ be the unique radially symmetric solution of \eqref{f-eqn} which satisfies  \eqref{gro inf} with $A=1$ given by Theorem \ref{f-existence-uniqueness-thm},
\begin{equation}\label{f-lambda-defn}
f_\lambda(x)= \lambda^{\frac{2}{1-m}}f_1(\lambda x) \quad \forall x\in\R^n\setminus\{0\}, \lambda>0,
\end{equation}
\begin{equation}\label{def g lam}
g_\lambda(r)=r^{-\frac{n-2}{m}}f_\lambda(r^{-1})\quad\forall r>0, \lambda>0,
\end{equation}
\begin{equation}\label{def Ulam}
U_{\lambda}(x,t)=e^{-\alpha t}f_{\lambda}(e^{-\beta t}x)\quad \forall x\in \R^n \setminus \{0\},t\in \R, \lambda>0,
\end{equation}
and
\begin{equation}\label{def Ulam inv}
\overline U_\lambda(r ,t)=e^{-\widetilde \alpha t} g_\lambda(e^{-\widetilde \beta t} r)\quad \forall x\in \R^n \setminus \{0\},t\in \R, \lambda>0,
\end{equation}
for the rest of the paper. Note that by \eqref{def g lam}, \eqref{def Ulam} and \eqref{def Ulam inv}.
\begin{equation}\label{U-lambda-U-bar-eqn}
\overline U_\lambda(r ,t)=r^{-\frac{n-2}{m}}U_\lambda(r^{-1},t)\quad \forall x\in \R^n \setminus \{0\},t\in \R, \lambda>0.
\end{equation}
Hence the definition \eqref{def Ulam inv} is consistent with the definition \eqref{u-bar-defn}.

\begin{remark}\label{flambda12a}
The function $f_\lambda$ satisfies \eqref{ode} with 
\begin{equation}\label{gro fl 0}
\lim_{r\to 0}\frac{r^2f_\lambda (r)^{1-m}}{\log r^{-1}}=\lim_{r\to 0}\frac{(\lambda r)^2f_1 (\lambda r)^{1-m}}{\log\lambda-\log (\lambda r)} =\frac{2(n-1)(n-2-nm)}{|\beta| (1-m)}
\end{equation}
and
\begin{equation}\label{gro fl inf} 
\lim_{r \to \infty} r^{\frac{n-2}{m}} f_\lambda(r)= \lim_{r \to \infty} \lambda^{\frac{2}{1-m}-\frac{n-2}{m}} (\lambda r)^{\frac{n-2}{m}} f_1(\lambda r)=\lambda^{\frac{2}{1-m}-\frac{n-2}{m}}.
\end{equation}
Hence by Theorem \ref{f-existence-uniqueness-thm}, $f_{\lambda}=f_{\beta,A}$ is the unique solution of \eqref{ode} which satisfies \eqref{gro 0} and \eqref{gro inf} with $A=\lambda^{-\gamma_1}$, where
\begin{equation}\label{gamma1-defn}
\gamma_1=\frac{n-2}{m}-\frac{2}{1-m}=\frac{n-2-nm}{m(1-m)}>0.
\end{equation}
Note that since $f_1$ satisfies \eqref{gro inf} with $A=1$,  there exists a constant $r_1>0$ such that
\begin{equation*}
\frac{1}{2}\le r^{\frac{n-2}{m}}f_1(r)\le \frac{3}{2}\quad\forall r\ge r_1.
\end{equation*}
Hence, by \eqref{f-lambda-defn},
\begin{equation*}
\frac{1}{2}\lambda^{\frac{2}{1-m}-\frac{n-2}{m}}\le r^{\frac{n-2}{m}}f_{\lambda}(r)\le \frac{3}{2}\lambda^{\frac{2}{1-m}-\frac{n-2}{m}}\quad\forall r\ge r_1/\lambda.
\end{equation*}
Moreover, by \eqref{gro fl 0}, there exist constants $c_2>c_1>0$ such that for any $\lambda>0$, there exists a constant $r_2=r_2(\lambda)>0$ such that
\begin{equation*}
c_1\left(\frac{\log r^{-1}}{r^2}\right)^{\frac{1}{1-m}}\le f_{\lambda}(r)\le c_2\left(\frac{\log r^{-1}}{r^2}\right)^{\frac{1}{1-m}}\quad\forall 0<r<r_2(\lambda).
\end{equation*}
\end{remark}

\begin{theorem}[Higher order asymptotic of self-similar solutions near the origin]\label{Thm HAB}
Let $n\ge3$, $0<m<\frac{n-2}{n}$, $\lambda>0$, $A>0$,  and $\alpha$, $\beta$, $\4{\alpha}$, $\4{\beta}$, $\gamma_1$ be as given in \eqref{ab}, \eqref{tilde ab}, and \eqref{gamma1-defn} respectively. Let 
\begin{equation}\label{a1-defn}
a_1:=\frac{(n-2-(n+2)m)^2}{4(n-2-nm)^2}-\frac{(1-m)^2 a_2(1, 1)}{4(n-1)(n-2-nm)^2}
\end{equation}
and 
\begin{equation}\label{k0-defn}
K_0=\frac{(1-m)K(1,1)}{2(n-1)(n-2-nm)},
\end{equation} 
where
\begin{align}\label{a2-defn}
a_2(\eta, \widetilde \beta)=&\frac{2(1-2m)(n-1)(n-2-nm)}{(1-m)^2}+ \frac{(n-1)(n-2-(n+2)m)^2}{(1-m)^2}\notag\\
&\quad -\frac{(n-2-(n+2)m)}{(1-m)}K(\eta, \widetilde \beta) \widetilde \beta
\end{align}
and
$K(\eta,\4{\beta})$ is given by \eqref{h1-infty-limit} 
for $m \neq \frac{n-2}{n+2}$ and by \eqref{h-infty-limit} for $m=\frac{n-2}{n+2}$ respectively. 
Let
\begin{equation}\label{a3-defn}
a_3(A,\4{\beta})=a_1+\frac{(n-2-(n+2)m)}{2(n-2-nm)\gamma_1}\log (A\4{\beta}^{\frac{1}{1-m}}).
\end{equation}
Let $f_{\beta,A}$ by the unique solution of \eqref{ode} which satisfies  \eqref{gro inf}
and let $f_{\lambda} (r)$ be as given by \eqref{f-lambda-defn}. Then  
\begin{align}\label{f-lambda-A-expansion}
f_{\beta,A}^{1-m}(r)
=&\frac{2(n-1)(n-2-nm)}{(1-m)|\beta| r^2}\left\{\log r^{-1} +\frac{(n-2-(n+2)m)}{2(n-2-nm)}\log (\log r^{-1})\right.\notag\\
&\qquad \left. +K_0+\frac{1}{\gamma_1}\log A +\frac{m} {n-2-nm} \log|\beta|+\frac{a_3(A,\4{\beta})}{\log r^{-1}} \right.\notag\\
&\qquad \left. +\frac{(n-2-(n+2)m)^2}{4(n-2-nm)^2}\cdot \frac{\log (\log r^{-1})}{\log r^{-1}}+ o\left(\frac{1}{\log r^{-1}}\right)\right \} \quad \text{as } r\to 0
\end{align}
and
\begin{align}\label{f-lambda-expansion}
f_{\lambda}^{1-m}(r)
=&\frac{2(n-1)(n-2-nm)}{(1-m)|\beta| r^2}\left\{\log r^{-1} +\frac{(n-2-(n+2)m)}{2(n-2-nm)}\log (\log r^{-1})\right.\notag\\
&\qquad \left. +K_0-\log\lambda +\frac{m} {n-2-nm} \log|\beta|+\frac{a_3(\lambda^{-\gamma_1},\4{\beta})}{\log r^{-1}} \right.\notag\\
&\qquad \left. +\frac{(n-2-(n+2)m)^2}{4(n-2-nm)^2}\cdot \frac{\log (\log r^{-1})}{\log r^{-1}}+ o\left(\frac{1}{\log r^{-1}}\right)\right \} \quad \text{as } r\to 0.
\end{align}
\end{theorem}

\begin{theorem}[$L^1$-contraction with weight $|x|^{-\mu}$]\label{u-v-weight-mu-L1-contraction}
Let $n\ge 3$, $0<m<\frac{n-2}{n}$, $\lambda_1> \lambda_2>0$, and $\alpha$, $\beta$ be as given by \eqref{ab}. Let 
$$\mu_1=n-\frac{2}{1-m}$$
and $f_{\lambda_i}$, $U_{\lambda_i}$, $i=1,2$ be as given by \eqref{f-lambda-defn} and \eqref{def Ulam} respectively with $\lambda=\lambda_1,\lambda_2$. Let $u_{0,1}$, $u_{0,2}$ satisfy
\begin{equation}\label{initial_condition_lower-upper-bd}
f_{\lambda_1} (x)\le u_{0,i}(x)\le f_{\lambda_2}(x) \quad \text{ in } \re^n \setminus \{0\}
\quad\forall i=1,2
\end{equation}
and let $u_1$, $u_2$ be the solutions of \eqref{main-eq} with initial values $u_{0,1}$, $u_{0,2}$, respectively which satisfy 
\begin{equation}\label{uvU}
U_{\lambda_1} (x,t)\le u_i(x,t)\le U_{\lambda_2}(x,t) \quad \text{ in } (\re^n \setminus \{0\}) \times (0, \infty)\quad\forall i=1,2.
\end{equation} 
Suppose that 
\begin{equation}\label{u0-v0-l1-weighted}
u_{0,1}-u_{0,2}\in L^1 \left(|x|^{-\mu}; \re^n \setminus \{0\} \right)
\end{equation}
holds for some constant $\mu\in (0,\mu_1]$. Then for 
\begin{equation}\label{mu-m-range}
\mu<\mu_1\quad \text{ or }\quad\left\{\begin{aligned}
&\mu=\mu_1\\
&0<m<\min\left(\frac{n-2}{n},\frac{1}{2}\right),\end{aligned}\right.
\end{equation}
\begin{equation}\label{uu2}
\int_{\re^n \setminus \{0\}}|u_1-u_2|(x,t) |x|^{-\mu} \, dx \le \int_{\re^n \setminus \{0\}}|u_{0,1}-u_{0,2}|(x) |x|^{-\mu} \, dx \quad \forall t>0
\end{equation}
and
\begin{equation}\label{u+u2}
\int_{\re^n \setminus \{0\}} (u_1-u_2)_+(x,t) |x|^{-\mu} \, dx \le \int_{\re^n \setminus \{0\}}(u_{0,1}-u_{0,2})_+(x) |x|^{-\mu}\, dx\quad \forall t>0
\end{equation}
hold. 
\end{theorem}

\begin{theorem}[Existence and uniqueness of solution]\label{thm exi uni u}
Let $n\ge 3$, $0<m<\frac{n-2}{n}$,  and $\alpha$, $\beta$ be as given by \eqref{ab}. Let $\lambda_1> \lambda_2>0$ and $f_{\lambda_i}$, $U_{\lambda_i}$, $i=1,2$ be as given by \eqref{f-lambda-defn} and \eqref{def Ulam} respectively with $\lambda=\lambda_1,\lambda_2$. Suppose $u_0$ satisfy 
\begin{equation}\label{ini con}
f_{\lambda_1} \le u_0 \le f_{\lambda_2}\quad\mbox{ in }\re^n \setminus \{0\}.
\end{equation}
Then \eqref{main-eq} has a unique solution $u$  which satisfies 
\begin{equation}\label{utu}
u_t \le \frac{u}{(1-m) t} \quad \text{ in } (\re^n \setminus \{0\}) \times (0, \infty)
\end{equation}
and
\begin{equation}\label{Ulam1u ent}
U_{\lambda_1}(x,t) \le u(x,t) \le U_{\lambda_2}(x,t) \quad \text{ in } (\re^n \setminus \{0\}) \times (0, \infty).
\end{equation} 
Hence if $u_0$ is radially symmetric, then for any $t>0$, $u(x, t)$ is radially symmetric in $x\in\R^n\setminus\{0\}$.
\end{theorem}

\begin{remark}
If $u_{0,1}$, $u_{0,2}$, $u_1$, $u_2$, and $\mu$ are as given by Theorem \ref{u-v-weight-mu-L1-contraction}, then by  Theorem \ref{u-v-weight-mu-L1-contraction}, 
\begin{align}\label{rmm exp}
\int_{\re^n \setminus \{0\}}|\widetilde {u}_1-\widetilde {u}_2|(x,t) |x|^{- \mu} \, dx
=&e^{( \alpha - \beta n+ \beta \mu )t} \int_{\re^n \setminus \{0\}} |u_1-u_2|(y,t) |y|^{- \mu} \, dy\quad\forall t>0,0<\mu\le\mu_1\notag\\
\le&e^{( \alpha - \beta n+\beta \mu )t} \int_{\re^n \setminus \{0\}}|\widetilde{u}_{0,1}-\widetilde{u}_{0,2}|(x) |x|^{- \mu} \, dx\quad\forall t>0,0<\mu\le\mu_1.
\end{align}
Since $\mu\in (0,\mu_1]$,  $\alpha - \beta n+\beta \mu \ge 0$. Hence one does not know whether  the left hand side of \eqref{rmm exp} will converge to zero as $t$ goes to infinity. Thus for the asymptotic large time behaviour of the solutions of \eqref{main-eq}, we need another $L^1$-contraction result for the solutions of \eqref{main-eq} with weight $f^{m \gamma}_\lambda$ for some constant $\gamma>0$ that will imply the difference of the rescaled solutions of \eqref{main-eq} in the weighted $L^1$-norm converges to zero as $t$ goes to infinity.
\end{remark}

\begin{theorem}\label{u-v-L1-contraction2}
Let $n=3,4$, $\frac{n-2}{n+2}\le m<\frac{n-2}{n}$, 
\begin{equation}\label{gamma2-defn}
\gamma_2 :=\frac{1-m}{2m}\left(n-\frac{2}{1-m}\right),
\end{equation}
 and $\alpha$, $\beta$ be as given by \eqref{ab}.
Let $\lambda_1> \lambda_2>0$, $\lambda_3>0$, and $f_{\lambda_i}$, $U_{\lambda_i}, i=1,2,3$, be as given by \eqref{f-lambda-defn} and \eqref{def Ulam} respectively with $\lambda=\lambda_1,\lambda_2, \lambda_3$. Let $u_{0,1}$, $u_{0,2}$ satisfy \eqref{initial_condition_lower-upper-bd},
\begin{equation}\label{u0-v0-weighted-l1}
u_{0,1}-u_{0,2}\in L^1\left(f_{\lambda_3}^{m\gamma_2}; \re^n \setminus \{0\}\right)
\end{equation}
and let $u_1$, $u_2$ be the solutions of \eqref{main-eq} with initial values $u_{0,1}$, $u_{0,2}$, respectively which satisfy \eqref{uvU}. Then
\begin{equation}\label{uu0a}
\int_{\R^n\setminus\{0\}} |u_1-u_2|(x,t) f_{\lambda_3}^{m \gamma_2} (x)\, dx \le \int_{\R^n\setminus\{0\}} |u_{0,1}-u_{0,2}|(x) f_{\lambda_3}^{m \gamma_2} (x)\, dx \quad \forall t>0.
\end{equation}
\end{theorem}

We note that when $\frac{n-2}{n+2}\le m<\frac{n-2}{n}$,  \eqref{uu0a} implies 
\begin{equation*}
\int_{\R^n\setminus\{0\}}|\widetilde u_1- \widetilde u_2|(y,t) f_{\lambda_3}^{m \gamma_2}(y)\,dy\le \int_{\R^n\setminus\{0\}} |u_{0,1}(y)-u_{0,2}(y)| f_{e^{-\beta t} \lambda_3}^{m \gamma_2}(y)\,dy \quad \forall t>0
\end{equation*}
and the right hand side converges to zero as $t$ goes to infinity, since by Remark \ref{flambda12} $f_{\lambda}(x)$ converges to zero for any $0\ne x\in\R^n$ as $\lambda\to\infty$.

\begin{theorem}[Asymptotic behaviour of solutions]\label{thm assm}
Let $n=3,4$, $\frac{n-2}{n+2}\le m<\frac{n-2}{n}$, 
 and $\alpha$, $\beta$, $\gamma_2$ be as given by \eqref{ab} and \eqref{gamma2-defn} respectively.
 Let $\lambda_1\ge \lambda_0 \ge \lambda_2>0$, $\lambda_3>0$, and $f_{\lambda_i}$, $i=0,1,2,3$ be as given by \eqref{f-lambda-defn} with $\lambda=\lambda_0,\lambda_1,\lambda_2,\lambda_3$. 
Let $u_0$ satisfy \eqref{ini con} and 
\begin{equation}\label{u0-f0-weighted-l1}
u_0-f_{\lambda_0} \in L^1 \left(f^{m\gamma}_{\lambda_3} ; \R^n\setminus\{0\} \right).
\end{equation}
Let $u$ be the solution of \eqref{main-eq} which satisfies \eqref{Ulam1u ent} and let $\4{u}$ be given by \eqref{rescaled-soln-defn}. 
Then $\widetilde u(\cdot,t)$ converges uniformly in $C^2(K)$ for any compact subset $K$ of $\re^n \setminus \{0\}$ to $f_{\lambda_0}$ as $t\to \infty$ and 
\begin{equation}\label{u-tilde-l1-convergence}
\widetilde u(\cdot, t) \to f_{\lambda_0}\quad\mbox{ in }L^1(f^{m\gamma}_{\lambda_3};\re^n \setminus\{0\})\quad\mbox{ as }t \to \infty.
\end{equation}
\end{theorem}

\begin{theorem}[Asymptotic behaviour of radially symmetric solutions]\label{thm assm rad sym}
Let
\begin{equation}\label{m-n-relation2}
3\le n<8\quad\mbox{ and }\quad 1-\sqrt{\frac{2}{n}}\le m<\min\left(\frac{2(n-2)}{3n},\frac{n-2}{n+2}\right),
\end{equation}
and $\alpha$, $\beta$, $\4{\alpha}$, $\4{\beta}$ be as given by \eqref{ab} and \eqref{tilde ab} respectively and 
\begin{equation}\label{gamma3-defn}
\gamma_3:=\frac{1}{m}\left(\frac{n\4{\beta}}{\4{\alpha}}-1\right)=\frac{n(1-m)}{n-2-nm}-\frac{1}{m}
\end{equation}
Let $\lambda_1\ge \lambda_0 \ge \lambda_2>0$, $\lambda_3>0$, and $f_{\lambda_i}$, $i=0,1,2,3$ be as given by \eqref{f-lambda-defn} with $\lambda=\lambda_0,\lambda_1,\lambda_2,\lambda_3$. Let $u_0$ be a radially symmetric function that satisfy \eqref{ini con} and 
\begin{equation}\label{u0-f0-weighted-l1-new}
u_0-f_{\lambda_0} \in L^1 \left(|x|^{\frac{n-2}{m} +(n-2) \gamma_3 -2n} f^{m\gamma_3}_{\lambda_3} ; \re^n \setminus \{0\} \right).
\end{equation}
Let $u$ be the solution of \eqref{main-eq} which satisfies \eqref{Ulam1u ent} and let $\4{u}$ be given by \eqref{rescaled-soln-defn}.  Then $\widetilde u(\cdot,t)$ converges uniformly in $C^2(K)$ for any any compact subset $K\subset \re^n \setminus \{0\}$ to $f_{\lambda_0}$ 
as $t\to \infty$ and 
\begin{equation}\label{u-tilde-l1-convergence2}
\widetilde u(\cdot, t) \to f_{\lambda_0}\quad\mbox{ in }L^1(|x|^{\frac{n-2}{m}+(n-2)\gamma_3 -2n} f_{\lambda_3}^{m \gamma_3} ; \re^n \setminus \{0\})\quad\mbox{ as }t \to \infty.
\end{equation}
\end{theorem}

\section{Existence and uniqueness of radially symmetric eternal self-similar solutions}\label{section-eternal-self-similar-soln}
\setcounter{equation}{0}
\setcounter{theorem}{0}

In this section, by using the inversion method of K.M.~Hui and Soojung Kim \cite{HK}, we will prove the existence of solution $f$ of \eqref{ode} which blow up at the origin with $\alpha$, $\beta$ as given by \eqref{ab}. We will also prove the growth estimates \eqref{gro 0} and \eqref{gro inf} for such solution $f$ and the uniqueness of such solution.
We first recall a result of \cite{HK}.

\begin{theorem}[Lemma 2.1 and Theorem 2.4 of \cite{HK}]\label{g-existence-thm} 
Let $n \ge 3$, $0<m<\frac{n-2}{n}$, $\widetilde \alpha>0,$ $\widetilde \beta>0,$ $\frac{\widetilde \alpha}{\widetilde \beta}\le \frac{n-2}{m}$, and $\eta>0.$
\begin{enumerate}[(a)] 
\item If $0<m<\frac{n-2}{n+1}$, then there exists a unique solution $g\in C^1([0,\infty);\R)\cap C^2((0,\infty);\R)$ of \eqref{inv eq} which satisfies 
\begin{equation}\label{g-at-x=0}
g(0)=\eta\quad\mbox{ and }\quad g_r(0)=0. 
\end{equation}
\item If $\frac{n-2}{n+1}\leq m< \frac{n-2}{n}$, then there exists a unique solution $g\in C^{0, \delta_0}([0,\infty);\R)\cap C^2((0,\infty);\R)$ of \eqref{inv eq} which satisfies 
\begin{equation}\label{eq-initial-m-large}
g(0)=\eta\quad \mbox{ and } \quad \lim_{r\to 0^+}r^{\delta_1} g_r(r)=-\frac{\widetilde\alpha \eta^{2-m}}{n-2-2m},
\end{equation} 
where 
\begin{equation}\label{defn-delta0-1}
\delta_1=1-\frac{n-2-nm}{m}\in[0,1)\quad\mbox{ and }\quad\delta_0=\frac{1-\delta_1}{2}=\frac{n-2-nm}{2m}\in(0,1/2].
\end{equation} 
Moreover 
\begin{equation}\label{g-monotone-expression}
g(r)+\frac{\widetilde \beta}{\widetilde \alpha} rg_r(r) >0 \quad\forall r>0
\end{equation}
and hence
\begin{equation}\label{gm-superharmonic}
(g^m)_{rr}+\frac{n-1}{r}(g^m)_r<0\quad\forall r>0.
\end{equation}
\end{enumerate}
\end{theorem}

Unless otherwise stated, we will use $g=g_{\widetilde \beta, \eta} (r)$ to denote the unique solution of 
\begin{align}\label{g-eqn}
\left\{\begin{aligned}
&\frac{n-1}{m}\left((g^m)_{rr}+ \frac{n-1}{r}(g^m)_r \right)+ r^{\frac{n-2}{m} -n-2}\left ( \widetilde \alpha g+\widetilde \beta r g_r \right )=0, \quad g>0, \quad\mbox{ in }(0,\infty),\\
&g(0)=\eta,\end{aligned}\right.
\end{align}
for some constant $\eta>0$ given by Theorem \ref{g-existence-thm} and let
\begin{equation}\label{w-defn}
w(r):=r^{(1-m)\frac{\widetilde \alpha}{\widetilde \beta}}g(r)^{1-m}.
\end{equation}
By abuse of notation, we will also let $g(x)$ denote the radially symmetric function on $\R^n\setminus\{0\}$ with value $g(|x|)$.

\begin{lem}\label{h1 w}
$$ w_r(r)>0 \quad \forall r>0.$$
\end{lem}
\begin{proof}
By \eqref{g-monotone-expression},
\begin{equation}\label{w'-eqn}
w_r(r)=(1-m)\frac{\widetilde \alpha}{\widetilde \beta} r^{\frac{\widetilde \alpha}{\widetilde \beta}(1-m)-1}g(r)^{-m}\left(g(r)+\frac{\widetilde \beta}{\widetilde \alpha} rg_r(r)\right)>0 \quad \forall r>0
\end{equation}
and the lemma follows.
\end{proof}

Let $q(r)=r^{\widetilde \alpha/\widetilde \beta}g(r)$, $s=\log r$, and $\widetilde q(s)=q(r)$. Then
$$
r^{2}(q^m)_{rr}+\left(n-1-\frac{2m\widetilde \alpha}{\widetilde \beta}\right) r(q^m)_r-\frac{m\widetilde \alpha}{\widetilde \beta}\left(n-2-\frac{m\widetilde \alpha}{\widetilde \beta}\right)q^m+\frac{m\widetilde \beta}{n-1} r q_r=0
\quad\forall r>0
$$
and
\begin{equation}\label{tl q ms}
(\widetilde q^m)_{ss}+\left(n-2-\frac{2m\widetilde \alpha}{\widetilde \beta}\right)(\widetilde q^m)_s-\frac{m\widetilde \alpha}{\widetilde \beta}\left(n-2-\frac{m\widetilde \alpha}{\widetilde \beta}\right)\widetilde q^m+\frac{m\widetilde \beta}{n-1} \widetilde q_s=0\quad\mbox{ in }\R.
\end{equation}
Let $\widetilde w(s)=\widetilde q^{1-m}(s)$. Then $w(r)=\widetilde w(s)$ and
\begin{equation}\label{ws}
\widetilde w_{ss}= \frac{1-2m}{1-m}\cdot \frac{\widetilde w_s^2}{\widetilde w}-\frac{(n+2)m-(n-2)}{1-m}\widetilde w_s+\frac{2(n-2-nm)}{1-m}\widetilde w-\frac{\widetilde \beta}{n-1} \widetilde w \widetilde w_s\quad\mbox{ in }\R
\end{equation}
 Hence
\begin{equation}\label{wr}
w_{rr}+\left(n-1-\frac{2m\widetilde \alpha}{\widetilde \beta}\right) \frac{w_r}{r}-\frac{1-2m}{1-m} \frac{w_r^2}{w}+\frac{\widetilde \beta}{n-1} \frac{w w_r}{r}
-\frac{2(n-2-nm)}{1-m}\frac{ w}{r^2}=0\quad\forall r>0.
\end{equation}
Let
\begin{equation}\label{b0b1}
b_0=\left(n-2-\frac{2\widetilde\alpha m}{\widetilde\beta}\right)=\frac{(n+2)m-(n-2)}{1-m}\quad \mbox{ and }\quad b_1=\frac{2(n-2-nm)}{1-m}.
\end{equation}

Since \eqref{ws} and \eqref{wr} are of the same form as (3.7) and (3.8) of \cite{Hs1}, by the same argument as the proof of Lemma 3.2 of \cite{Hs1} but with \eqref{ws} and \eqref{wr} replacing (3.7), (3.8) in the proof there, the following result is obtained. We briefly introduce the proof of the following lemma, for the reader's convenience.

\begin{lem}\label{c123}
There exist positive constants $C_1, C_2$ and $C_3$ such that
\begin{equation}\label{rwrc1}
\frac{rw_r(r)}{w(r)} \le C_1, \quad \forall r \ge0
\end{equation}
and
\begin{equation}\label{c2rwr}
C_2\le rw_r(r) \le C_3, \quad \forall r \ge 1.
\end{equation}
Moreover
\begin{equation}\label{wti}
w(r) \to \infty \quad \text{ as } r \to \infty.
\end{equation}
\end{lem}

\begin{proof}
First, we note that $\lim_{s\to -\infty} q(s)=0$, $ \lim_{s\to -\infty} \tilde q_s(-\infty)=0$ and by Lemma \ref{h1 w}, $\tilde q_s >0$ on $(-\infty, \infty)$. If $b_0 \ge 0$, then by \eqref{tl q ms}, 
\begin{equation*}
(\tilde q^m)_{ss} -\hat b_1 \tilde q^m \le 0 \Rightarrow (\tilde q^m)_s \le b_1 \tilde q^m \Rightarrow \frac{\tilde q_s }{\tilde q}\le \frac{b_1}{m}\Rightarrow \frac{rw_s}{w}\le \frac{(1-m)b_1}{m}, \quad \forall r\ge 0
\end{equation*}
and \eqref{rwrc1} follows.

If $b_0<0$, by \eqref{tl q ms}, 
\begin{equation}\label{qb0}
(\tilde q^m)_{ss}+b_0 (\tilde q^m)_s- \hat b_1 q^m\le 0.
\end{equation}
Let $p=\frac{(\tilde q^m)'}{\tilde q^m}$. Then by \eqref{qb0}, 
\begin{equation}\label{p'}
p'=\frac{(\tilde q^m)''}{\tilde q^m}-\frac{\left((\tilde q^m)'\right)^2}{\tilde q^m}\le |b_0|p+\hat b_1-p^2=-\left(p-(|b_0|/2)\right)^2+\hat b_1+(b_0^2/4).
\end{equation}
Let
$$b_2=\max\left( \frac{3m}{2} \frac{\tilde \alpha}{\tilde \beta} , \sqrt{\hat b_1+b_0^2}+|b_0| \right).$$
We claim that
\begin{equation}\label{pb2}
p(s)\le b_2\quad \forall s\in \re.
\end{equation}
Suppose \eqref{pb2} does not hold. Then there exists $s_0\in \re$ such that $p(s_0)>b_2$. Since
\begin{equation}\label{pmq}
p(s)=m \frac{\tilde q'}{\tilde q}=\frac{m}{1-m} \frac{\tilde w'}{\tilde w}=\frac{m}{1-m} \frac{r w'}{w}=m \frac{\tilde \alpha}{\tilde \beta} \left(1+ \frac{\tilde \beta}{\tilde \alpha} \frac{rg'}{g} \right),
\end{equation}
$p(s)\to m \frac{\tilde \alpha}{\tilde \beta}$ as $s \to -\infty.$ Let $s_1=\inf\left \{s'<s_0 : p(s)>b_2, \forall s'\le s \le s_0\right \}$. Then, we know that $-\infty <s_1<s_0$, $p(s)>b_2$ for all $s\in (s_1,s_0)$ and $p(s_1)=b_2$. By \eqref{p'}, $p'<0$ for all $s\in (s_1,s_0)$. Hence $p(s_0) \le p(s_1)=b_2$. Thus contradiction arises and \eqref{pb2} follows. Then by \eqref{pb2} and \eqref{pmq}, \eqref{rwrc1} holds with $C_1=b_2\frac{1-m}{m}.$

The inequalities $\eqref{c2rwr}$ and the limit of $w$ \eqref{wti} are proved exactly the same argument as the proof of Lemma 3.2 of \cite{Hs1}, with constants $a_1=\frac{\tilde \alpha}{\tilde \beta}(1-m) \left (n-2-\frac{\tilde \alpha}{\tilde \beta}m\right )=\frac{m}{1-m}b_1>0$, $a_2=\frac{\tilde \beta}{(n-1) a_1}$, $a_3=a_1^{-1} \max \left(|b_0|, \frac{|1-2m|}{|1-m|w(1)}\right)$, $C'_2=\min \left(1, (2a_2)^{-1}, \tilde w(1)/(8a_3), \sqrt{\tilde w(1)/(8a_3)} \right)$, $a_4=\frac{\tilde \beta}{3(n-1)}$, and $a_5=\frac{\tilde \beta (1-m)}{3(n-1)}.$ Thus, we refere to \cite{Hs1} for the rest of the proof.
\end{proof}

\begin{theorem}\label{w-limit-thm}
\begin{equation}\label{inv gro rate}
\lim_{r\to \infty}\frac{\4{w}(s)}{s}=\lim_{r\to \infty}\frac{w(r)}{\log r}=\lim_{r\to \infty}rw_r=\lim_{s\to \infty}\4{w}_s(s)=\frac{2(n-1)(n-2-nm)}{(1-m)\4{\beta}}.
\end{equation}
\end{theorem}
\begin{proof}
We will use a modification of the proof of Theorem 1.3 of \cite{Hs1} to prove this theorem.
Let $b_0$ and $b_1$ be as given by \eqref{b0b1} and let $v(r)=rw_r(r)$
\begin{equation}\label{a0}
a_0=\frac{n-1}{\widetilde \beta}b_1, \text{ and } \quad v_1(r)=v(r)-a_0.
\end{equation}
By \eqref{wr} and a direct computation,
\begin{align}
&\left(r^{b_0}v(r)w^{\frac{2m-1}{1-m}}\right)_r=\frac{\4{\beta}}{n-1}\cdot\frac{w^{\frac{m}{1-m}}}{r^{1-b_0}}(a_0 - v(r))\quad\forall r>0\notag\\
\Rightarrow\quad &v_{1,r}+\frac{b_0}{r}v_1+\frac{\4{\beta}}{n-1}\frac{w(r)}{r}v_1
=\frac{1-2m}{1-m}\cdot\frac{v(r)^2}{rw(r)}-\frac{b_0a_0}{r}\quad\forall r>0.\label{v1-eqn} 
\end{align}
Let $\{r_i\}^{\infty}_{i=1}$ be a sequence of positive numbers such that $r_i \to \infty$ as $i \to \infty$. By Lemma \ref{c123}, there exist positive constants $C_1, C_2$ and $C_3$ such that \eqref{rwrc1} and \eqref{c2rwr} hold. Then by \eqref{c2rwr}, there exists a subsequence of the sequence $\{v(r_i)\}^{\infty}_{i=1}$, which we still denote by $\{v(r_i)\}^{\infty}_{i=1}$ such that the subsequence converges to some constant $v_\infty$ as $i\to\infty$. 
\begin{equation}\label{c2vic3}
C_2\le v_\infty\le C_3.
\end{equation}
Let 
$$
\psi_1(r)=\exp \left( \frac{\4{\beta}}{n-1}\int^r_{1}\rho^{-1}w(\rho) d \rho \right).
$$
Then by \eqref{v1-eqn},
\begin{align}\label{v1-eqn2}
r^{b_0}\psi_1(r)v_1(r)=&f_1(1)v_1(1)-a_0b_0 \int^r_1 \rho^{b_0-1}\psi_1(\rho)d\rho\notag\\
&\qquad +\frac{1-2m}{1-m} \int^r_1 \rho^{b_0-1}v(\rho)^2\psi_1(\rho)w(\rho)^{-1}d\rho \quad\forall r\ge 1.
\end{align}
By \eqref{wti} there exists a constant $r_0>1$ such that
\begin{equation*}
w(r)> \frac{n-1}{\4{\beta}} \left( |b_0| +2 \right)\quad\forall r\ge r_0.
\end{equation*}
Hence
\begin{align}\label{rb0-1}
&\psi_1(r)\ge\left(\frac{r}{r_0}\right)^{|b_0| +2}\quad\forall r\ge r_0\notag\\
\Rightarrow\quad&r^{b_0-1}\psi_1(r)\ge \frac{r^{|b_0| +b_0+1}}{r_0^{|b_0| +2}} \ge r_0^{b_0 -2}r\quad\forall r\ge r_0
\notag\\
\Rightarrow\quad&r^{b_0-1}\psi_1(r) \to \infty \quad \text{ and } \quad r \to \infty.
\end{align}
By \eqref{c2rwr}, \eqref{wti}, \eqref{rb0-1}, and the l'Hospital rule,
\begin{align}\label{integral-limit1}
&\liminf_{r\to \infty}\frac{r^{b_0-1} \psi_1(r)}{w(r)}
=\liminf_{r\to \infty}\frac{\left (b_0-1+\frac{\4{\beta}}{n-1}w(r) \right )r^{b_0-1}\psi_1(r)}{rw_r(r)}= \infty\notag\\
\Rightarrow\quad&\int^r_1\frac{r^{b_0-1}v(\rho)^2\psi_1(\rho)}{w(\rho)} \, d \rho \to \infty \quad \text{ as } r\to \infty.
\end{align}
On the other hand by \eqref{wti}, \eqref{rb0-1} and the l'Hospital rule, 
\begin{align}\label{int rhof1}
\lim_{r\to \infty}\frac{\int^r_{1} \rho^{b_0-1}\psi_1(\rho) d\rho}{r^{b_0}\psi_1(r)}=&\lim_{r\to \infty}\frac{ r^{b_0-1}\psi_1(r) }{b_0r^{b_0-1}\psi_1(r)+\frac{\4{\beta}}{n-1} r^{b_0-1} w(r)\psi_1(r)}\notag\\ 
=&\lim_{r\to \infty}\frac{1}{b_0+\frac{\4{\beta}}{n-1} w(r)}=0.
\end{align}
By \eqref{wti}, \eqref{c2vic3}, \eqref{v1-eqn2}, \eqref{rb0-1}, \eqref{integral-limit1}, \eqref{int rhof1},  and the l'Hospital rule,
\begin{align*}
\lim_{i\to \infty} v_1(r_i)&=\lim_{i\to \infty} \frac{\psi_1(1)v_1(1)-a_0b_0\int^{r_i}_{1} \rho^{b_0-1} \psi_1(\rho) d\rho+\frac{1-2m}{1-m}\int^{r_i}_{1} \frac{\rho^{b_0-1}v(\rho)^2 \psi_1(\rho)}{w(\rho)} d\rho} {r_i^{b_0}\psi_1(r_i)}\\
&=\frac{1-2m}{1-m} \lim_{i\to \infty} \frac{r_i^{b_0-1}v(r_i)^2 \psi_1(r_i)w(r_i)^{-1}} {b_0r^{b_0-1}\psi_1(r_i)+\frac{\4{\beta}}{n-1} r^{b_0-1} w(r_i)\psi_1(r_i)}\\
&=\frac{1-2m}{1-m} \lim_{i\to \infty} \frac{v(r_i)^2 w(r_i)^{-1}} {b_0+\frac{\4{\beta}}{n-1} w(r_i)}=0.
\end{align*}
Hence $\lim_{i\to \infty} v(r_i)=a_0$. Since the sequence $\{r_i\}^{\infty}_{i=1}$ is arbitrary, $\lim_{r\to \infty} v(r)=a_0$
and the theorem follows.
\end{proof}

\begin{theorem}\label{exi}
For any constant $A>0$, there exists a solution $f$ of \eqref{ode} which satisfies \eqref{gro 0}, \eqref{gro inf} and \eqref{f dec}.
\end{theorem}
\begin{proof}
Let $g$ be the unique solution of \eqref{g-eqn}
given by Theorem \ref{g-existence-thm} with $\eta=A$. Then the function $f$ given by 
\eqref{f-g-relation} satisfies \eqref{ode}. By \eqref{f-g-relation}, \eqref{alpha-beta-and-tilde-relation}, \eqref{g-eqn}  and Theorem \ref{w-limit-thm}, \eqref{gro 0} and \eqref{gro inf} hold.
By \eqref{ab}, \eqref{alpha-beta-and-tilde-relation} and  \eqref{g-monotone-expression}, 
\begin{align*}
\alpha f(r) + \beta r f_r(r) &= \beta \left( \frac{\alpha}{\beta} f (r) + r f_r(r) \right)\\
& = \beta r^{-\frac{n-2}{m}}\left( \left( \frac{\alpha}{\beta} - \frac{n-2}{m} \right) g(r^{-1}) - r^{-1} g_r (r^{-1}) \right)\\
& = \widetilde \beta r^{-\frac{n-2}{m}}\left( \frac{\widetilde \alpha}{\widetilde \beta}  g(r^{-1}) + r^{-1} g_r (r^{-1}) \right) > 0 \quad \forall r>0
\end{align*}
and \eqref{f dec} follows.
\end{proof}

We next observe that the proof of Lemma 3.2, Lemma 3.3, and Remark 2 of \cite{HK} still holds when $\alpha$, $\beta$ are  given by \eqref{ab}. Hence we have the following results.

\begin{lem}[cf. Lemma 3.2 and Lemma 3.3 of \cite{HK}] \label{lem gg} 
Let $n\ge 3$, $0<m<\frac{n-2}{n}$, and $\alpha$, $\beta$ be as given by \eqref{ab}. Let $f$ be a solution of \eqref{ode} which satisfies 
\eqref{gro inf} for some positive constant $A$. Let $g$, $\widetilde \alpha$, $\widetilde \beta$ be as given by \eqref{def g} and \eqref{tilde ab} respectively.
Then $f$ satisfies \eqref{f dec} and $g$ is equal to the solution of \eqref{g-eqn} given by Theorem \ref{g-existence-thm} with $\eta=A$.
\end{lem}

\begin{remark} [cf. Remark 2 of \cite{HK}] \label{flambda12}
Let  $f_\lambda$ be as given by \eqref{f-lambda-defn}. Then
$$
\frac{d}{d\lambda}f_\lambda(r)<0 \quad\forall r>0,\lambda>0.
$$
Moreover for any $\lambda_1>\lambda_2>0$, there exists a constant $c_0>0$ such that 
$$
c_0 f_{\lambda_2}(r) \le f_{\lambda_1}(r) <f_{\lambda_2}(r)\quad \forall r>0.
$$
\end{remark}

We are now ready for the proof of Theorem \ref{f-existence-uniqueness-thm}.

\begin{proof}[Proof of Theorem \ref{f-existence-uniqueness-thm}]
By Theorem \ref{exi}, for any $A>0$, there exists a solution $f=f_{\beta,A}$ of \eqref{ode} which satisfies \eqref{gro 0}, \eqref{gro inf} and \eqref{f dec}. It remains to prove the uniqueness of solution of \eqref{ode} that satisfies  \eqref{gro inf}. 
Suppose \eqref{ode} has two solutions $f_1$, $f_2$, that satisfy \eqref{gro inf}. Let $g_1$, $g_2$ be given by \eqref{def g} with 
$f=f_1,f_2$. By Lemma \ref{lem gg}  both $g_1$ and $g_2$ are the solution of \eqref{g-eqn} with $\eta=A$ given by Theorem \ref{g-existence-thm}. Hence by \eqref{def g} and Theorem \ref{g-existence-thm},
\begin{equation*}
g_1(r)=g_2(r)\quad\forall r\ge 0\quad\Rightarrow\quad f_1(r)=f_2(r)\quad\forall r>0
\end{equation*}
and the theorem follows.

\end{proof}

\section{Higher order asymptotics of eternal self-similar solutions near the origin}
\label{section-Higher order asymptotics}
\setcounter{equation}{0}
\setcounter{theorem}{0}
In this section, we will use a modification of the technique of \cite{CD} and \cite{Hs4} to prove the higher order asymptotic of the eternal self-similar solutions near the origin.
Since the proofs are similar to that of \cite{CD} and \cite{Hs4}, we will only sketch the proofs here.

By \eqref{ws},
\begin{equation}\label{ti w ss}
\4{w}_{ss}=\left(\frac{1-2m}{1-m}\right)\frac{\4{w}_s^2}{\4{w}}+\frac{n-2-(n+2)m}{1-m} \4{w}_s+\frac{\4{\beta}}{n-1} \left ( \frac{2(n-1)(n-2-nm)}{(1-m) \4{\beta}} -\4{w}_s \right )\4{w}
\quad\forall s\in\R.
\end{equation}
Note that this is the same as (2.2) of \cite{Hs4} if $\4{\beta}$ is replaced by the constant $\beta$ in the paper and the sign of the second term on the right hand side of \eqref{ti w ss} is change to negative sign.
Let 
\begin{equation}\label{def h}
h(s)=\4{w}(s)-\frac{2(n-1)(n-2-nm)}{(1-m)\4{\beta}}s.
\end{equation}
Then by \eqref{ti w ss},
\begin{equation}\label{hss}
h_{ss}+\left(\frac{2(n-2-nm)}{1-m}s + \frac{\4{\beta}}{n-1}h - \frac{n-2-(n+2)m}{1-m} \right) h_s\\
=\frac{1-2m}{1-m}\cdot \frac{\4{w}_s^2}{\4{w}}+ b_2\quad\forall s\in\R,
\end{equation}
where
$$
b_2=\frac{2(n-1)(n-2-nm)(n-2-(n+2)m)}{(1-m)^2\, \4{\beta}}.
$$

\begin{lem}\label{hs}
Let $n\ge3$, $0<m<\frac{n-2}{n}$, $m \neq \frac{n-2}{n+2}$ and $\tilde \alpha$, $\tilde \beta$ satisfy \eqref{alpha-beta-and-tilde-relation}. Then $h$ satisfies
\begin{equation}\label{h-s-infty-limit}
\lim_{s \to \infty} \frac{h(s)}{\log s}=\lim_{s \to \infty} s h_s(s)=\frac{(1-m)b_2}{2(n-2-nm)}=\frac{(n-1) \left( n-2 -(n+2)m \right) }{(1-m)\4{\beta}}.
\end{equation}
\end{lem}
\begin{proof}
We will use a modification of the proof of Lemma 2.3 of \cite{Hs4} to prove this lemma. We first observe that by Theorem \ref{w-limit-thm},
\begin{equation}\label{w-ratio-0}
\lim_{s\to\infty}\frac{\4{w}_s^2(s)}{\4{w}(s)}=\frac{\lim_{s\to\infty}\4{w}_s^2(s)}{\lim_{s\to\infty}s\cdot (\4{w}(s)/s)}=0.
\end{equation}
Then by \eqref{hss} and \eqref{w-ratio-0} for any $0<\3<|b_2|/2$ there exists a constant $s_1\in\R$ such that
\begin{equation}\label{h-ineqn}
b_2-\3\le h_{ss}+\left(\frac{2(n-2-nm)}{1-m}s+\frac{\4{\beta}}{n-1}h-\frac{n-2-(n+2)m}{1-m}\right)h_s\le b_2+\3\quad\forall s\ge s_1.
\end{equation}
Let 
\begin{equation}\label{psi-defn}
\psi (s)=\mbox{exp}\,\left(\frac{n-2-nm}{1-m}s^2+\frac{\4{\beta}}{n-1}\int_1^sh(z)\,dz-\frac{n-2-(n+2)m}{1-m}s\right).
\end{equation}
Multiplying  \eqref{h-ineqn} by $\psi$ and integrating over $(s_1,s)$,
\begin{equation}\label{h'-ineqn}
\frac{\psi (s_1)h_s(s_1)+(b_2-\3)\int_{s_1}^s\psi (z)\,dz}{s^{-1}\psi (s)}\le sh_s(s)\le\frac{\psi (s_1)h_s(s_1)+(b_2+\3)\int_{s_1}^s\psi (z)\,dz}{s^{-1}\psi (s)}
\end{equation}
holds for any $ s\ge s_1$. Note that by Theorem \ref{w-limit-thm}, $h(s)=o(s)$  and $h(s)/s\to 0$ as $s\to\infty$. Hence $\psi(s)\to\infty$ as $s\to\infty$. Thus by  the l'Hospital rule, we have
\begin{equation}\label{psi-expression-ratio-limit}
\lim_{s\to\infty}\frac{\psi (s)}{s}=\lim_{s\to\infty}\psi(s)\left(\frac{2(n-2-nm)}{1-m}s+\frac{\4{\beta}}{n-1}h(s)-\frac{n-2-(n+2)m}{1-m}\right)=\infty
\end{equation}
and
\begin{align}\label{ratio-limit2}
\lim_{s\to\infty}\frac{\int_{s_1}^s\psi (z)\,dz}{s^{-1}\psi (s)}
=&\lim_{s\to\infty}\frac{\psi (s)}{-s^{-2}\psi (s)+s^{-1}\psi (s)\left(\frac{2(n-2-nm)}{1-m}s+\frac{\4{\beta}}{n-1}h(s)-\frac{n-2-(n+2)m}{1-m}\right)}\notag\\
=&\frac{(1-m)}{2(n-2-nm)}.
\end{align}
Letting first $s\to\infty$ and then $\3\to 0$ in \eqref{h'-ineqn}, by \eqref{psi-expression-ratio-limit} and  \eqref{ratio-limit2} we get \eqref{h-s-infty-limit} and the lemma follows.
\end{proof}

We introduce the result for the case $m=\frac{n-2}{n+2}$ in Proposition 3.1 of \cite{Hs4} and Proposition 2.3 of \cite{CD}.

\begin{lem}\cite{CD, Hs4}\label{h-limit2}
Let $n\ge3$, $0<m<\frac{n-2}{n}$, $m=\frac{n-2}{n+2}$ and $\tilde \alpha$, $\tilde \beta$ satisfy \eqref{alpha-beta-and-tilde-relation}. Then 
\begin{equation}\label{h'-limit5}
\lim_{s \to \infty}s^2h_s(s)=\frac{(n-1)(1-2m)}{(1-m)\4{\beta}}
\end{equation}
holds.
\end{lem}
\begin{proof}
Let $\psi$ be as given by \eqref{psi-defn}.
Since $m=\frac{n-2}{n+2}$, multiplying  \eqref{hss} by $\psi$ and integrating over $(1,s)$,
\begin{equation}\label{h'-ineqn6}
s^2h_s(s)=\frac{\psi (s_1)h_s(s_1)+\frac{(1-2m)}{(1-m)}\int_1^s\frac{\4{w}_s^2(z)}{\4{w}(z)}\psi (z)\,dz}{s^{-2}\psi (s)}\quad\forall s\ge 1.
\end{equation}
Hence by \eqref{psi-expression-ratio-limit} and the l'Hospital rule,
\begin{equation}\label{f-frac-limit1}
\lim_{s\to\infty}\frac{s^2}{\psi (s)}=\lim_{s\to\infty}\frac{2s}{\psi (s)\left(\frac{2(n-2-nm)}{1-m}s+\frac{\4{\beta}}{n-1}h(s)\right)}=0.
\end{equation}
By Theorem \ref{w-limit-thm} and the l'Hospital rule,
\begin{align}\label{f-frac-limit2}
\lim_{s\to\infty}\frac{\int_1^s\frac{\4{w}_s^2(z)}{\4{w}(z)}\psi (z)\,dz}{s^{-2}\psi (s)}
=&\lim_{s\to\infty}\frac{\frac{\4{w}_s^2(s)}{\4{w}(s)}\psi (s)}{-2s^{-3}\psi (s)+s^{-2}\psi (s)\left(\frac{2(n-2-nm)}{1-m}s+\frac{\4{\beta}}{n-1}h(s)\right)}\notag\\
=&\frac{1-m}{2(n-2-nm)}\lim_{s\to\infty}\frac{\4{w}_s^2(s)}{\4{w}(s)/s}\notag\\
=&\frac{1-m}{2(n-2-nm)}\lim_{s\to\infty}\4{w}_s\notag\\
=&\frac{n-1}{\4{\beta}}.
\end{align}
Letting $s\to\infty$ in \eqref{h'-ineqn6}, by \eqref{f-frac-limit1} and \eqref{f-frac-limit2}
we get \eqref{h'-limit5} and the lemma follows.
\end{proof}

For $m\ne\frac{n-2}{n+2}$, let 
\begin{equation}\label{def h1}
h_1(s)=h(s)-\frac{(n-1) \left( n-2 -(n+2)m \right) }{(1-m)\widetilde \beta} \log s.
\end{equation}
Then, $h_1$ satisfies
\begin{equation}\label{h1ss}
\begin{aligned}
h_{1,ss}&+\left(\frac{2(n-2-nm)}{(1-m)}s + \frac{\4{\beta}}{n-1}h- \frac{n-2-(n+2)m}{1-m} \right) h_{1,s}\\
&=\frac{1-2m}{1-m}\cdot \frac{\4{w}_s^2}{\4{w}}+a_3\left [\frac{1}{s^2}- \left(\frac{\4{\beta}}{n-1}h- \frac{n-2-(n+2)m}{1-m} \right) \frac{1}{s} \right] \quad \text{ in } \re,
\end{aligned}
\end{equation}
where $$a_3=\frac{(n-1) \left( n-2 -(n+2)m \right) }{(1-m)\widetilde \beta}.$$

\begin{lem}\label{h1s}
Let $n\ge3$, $0<m<\frac{n-2}{n}$, $m \neq \frac{n-2}{n+2}$  and $\tilde \alpha$, $\tilde \beta$ satisfy \eqref{alpha-beta-and-tilde-relation}. Then
\begin{equation}\label{h1'-limit}
\underset{\substack{s \to \infty}}{\lim}\frac{s^2 h_{1,s}(s)}{\log s}=- \frac{(n-1) \left( n-2 -(n+2)m \right)^2 }{2(n-2-nm)(1-m)\widetilde \beta}.
\end{equation}
\end{lem}
\begin{proof} 
We will use a modification of  the proof of Lemma 2.5 of \cite{Hs4} to prove this lemma. Let 
\begin{equation*}
\psi_2(s)=\frac{1-2m}{1-m}\cdot \frac{\4{w}_s^2}{\4{w}}+a_3\left [\frac{1}{s^2}- \left(\frac{\4{\beta}}{n-1}h- \frac{n-2-(n+2)m}{1-m} \right) \frac{1}{s} \right].
\end{equation*}
Then by  Theorem \ref{w-limit-thm} and Lemma \ref{hs},
\begin{equation}\label{H-infty-limit}
\lim_{s\to\infty}\frac{s\psi_2(s)}{\log s}=-\frac{(n-2-(n+2)m)}{1-m}a_3=-a_4,
\end{equation}
where
\begin{equation}
a_4=\frac{(n-1)(n-2-(n+2)m)^2}{(1-m)^2\4{\beta}}.
\end{equation}
By \eqref{h1ss} and \eqref{H-infty-limit} for any $0<\3<a_4/2$ there exists a constant $s_1>1$ such that
\begin{align}\label{h1-ineqn}
 (-a_4-\3)\frac{\log s}{s}\le &h_{1,ss}+\left(\frac{2(n-2-nm)}{1-m}s+\frac{\4{\beta}}{n-1}h-\frac{n-2-(n+2)m}{1-m}\right)h_{1,s}\notag\\
\le &(-a_4+\3)\frac{\log s}{s}\quad\forall s\ge s_1.
\end{align}
Let $\psi$ be given by \eqref{psi-defn}. Multiplying \eqref{h1-ineqn} by $\psi$ and integrating over $(s_1,s)$,
\begin{align}\label{h1'-ineqn3}
&\left(\frac{\psi(s_1)h_{1,s}(s_1)+(-a_4-\3)\int_{s_1}^s\frac{\log z}{z}\psi(z)\,dz}{s^{-2}\psi(s)\log s}\right)\notag\\
\le& \frac{s^2h_{1,s}(s)}{\log s}
\le\left(\frac{\psi(s_1)h_{1,s}(s_1)+(-a_4+\3)\int_{s_1}^s\frac{\log z}{z}\psi(z)\,dz}{s^{-2}\psi(s)\log s}\right)\quad\forall s\ge s_1.
\end{align}
By the l'Hospital rule and Lemma \ref{hs},
\begin{align}\label{psi-expression-ratio-limit2}
&\lim_{s\to\infty}\frac{\int_{s_1}^s\frac{\log z}{z}\psi(z)\,dz}{s^{-2}\psi(s)\log s}\notag\\
=&\lim_{s\to\infty}\frac{\frac{\log s}{s}\psi(s)}{s^{-3}\psi (s)(1-2\log s)+s^{-2}\psi(s)\left(\frac{2(n-2-nm)}{1-m}s+\frac{\4{\beta}}{n-1}h(s)-\frac{n-2-(n+2)m}{1-m}\right)\log s}\notag\\
=&\frac{1-m}{2(n-2-nm)}.
\end{align}
Letting first $s\to\infty$ and then $\3\to 0$ in \eqref{h1'-ineqn3}, by \eqref{psi-expression-ratio-limit2} we get \eqref{h1'-limit}
and the lemma follows.
\end{proof}

\begin{cor}\label{h1}
Let $n\ge3$, $0<m<\frac{n-2}{n}$, $m \neq \frac{n-2}{n+2}$ and $\tilde \alpha$, $\tilde \beta$ satisfy \eqref{alpha-beta-and-tilde-relation}. Then
\begin{equation}\label{h1-infty-limit}
K(\eta, \tilde \beta):= \lim_{s\to \infty} h_1(s) \in \re \quad \text{ exists}
\end{equation}
and
\begin{equation}\label{h1-expansion}
\begin{aligned}
h_1(s)=&K(\eta, \tilde \beta)+\frac{(n-1) \left( n-2 -(n+2)m \right)^2 }{2(n-2-nm)(1-m)\tilde \beta}\left(\frac{1+\log s}{s} \right)+o\left(\frac{1+\log s}{s} \right)\\
& \text{ as } s \to \infty.
\end{aligned}
\end{equation}
\end{cor}
\begin{proof}
We will use a modification of the proof of Corollary 2.6 of \cite{Hs4} to prove this lemma.
By Lemma \ref{h1s} there exists a constant $C_1>0$ such that
\begin{align}\label{h1-uniform-bd}
&\left|\frac{s^2h_{1,s}(s)}{\log s}\right|\le C_1\quad\forall s\ge 2\notag\\
\Rightarrow\quad&|h_1(s_1)-h_1(s_2)|\le\int_{s_1}^{s_2}|h_{1,s}(z)|\,dz\le C_1\int_{s_1}^{s_2}\frac{\log z}{z^2}\,dz\le C\int_{s_1}^{s_2}\frac{1}{z^{3/2}}\,dz\le\frac{C'}{\sqrt{s_1}}
\end{align}
for any $s_2>s_1\ge 2$. Hence \eqref{h1-infty-limit} holds. Then by Lemma \ref{h1s},
\begin{align*}
K(\eta, \tilde \beta)-h_1(s)=&\int_s^{\infty}h_{1,s}(z)\,dz\notag\\
=&\left(-\frac{(n-1) \left( n-2 -(n+2)m \right)^2 }{2(n-2-nm)(1-m)\tilde \beta}+o(1)\right)\int_s^{\infty}\frac{\log z}{z^2}\,dz\notag\\
=&\left(-\frac{(n-1) \left( n-2 -(n+2)m \right)^2 }{2(n-2-nm)(1-m)\tilde \beta}+o(1)\right)\left(\frac{1+\log s}{s} \right)\quad\mbox{ as }s\to\infty
\end{align*}
and \eqref{h1-expansion} follows.
\end{proof}

\begin{cor}\label{h-limit6-lem}
Let $n\ge3$, $0<m<\frac{n-2}{n}$, $m=\frac{n-2}{n+2}$ and $\tilde \alpha$, $\tilde \beta$ satisfy \eqref{alpha-beta-and-tilde-relation}. Then 
\begin{equation}\label{h-infty-limit}
K(\eta, \tilde \beta):= \lim_{s\to \infty} h(s) \in \re \quad \text{ exists}
\end{equation}
and
\begin{equation}\label{h-expansion-infty1}
h(s)=K(\eta,\4{\beta})-\frac{(n-1)(1-2m)}{(1-m)\4{\beta}}\cdot\frac{1}{s}+o(s^{-1})\quad\mbox{ as }s\to\infty.
\end{equation}
\end{cor}
\begin{proof}
By Lemma \ref{h-limit2} and an argument similar to the proof of Corollary \ref{h1}, 
\eqref{h-infty-limit} holds. Then by Lemma \ref{h-limit2},
\begin{align*}
K(\eta, \tilde \beta)-h(s)=&\int_s^{\infty}h_s(z)\,dz\notag\\
=&\left(\frac{(n-1)(1-2m)}{(1-m)\tilde{\beta}}+o(1)\right)\int_s^{\infty}\frac{1}{z^2}\,dz\notag\\
=&\left(\frac{(n-1)(1-2m)}{(1-m)\tilde{\beta}}+o(1)\right)\frac{1}{s}\quad\mbox{ as }s\to\infty
\end{align*}
and \eqref{h-expansion-infty1} follows.
\end{proof}

For $m\ne\frac{n-2}{n+2}$, let
\begin{equation}\label{h2-defn}
h_2(s)=h_1(s)-K(\eta, \widetilde \beta)-\frac{(n-1) \left( n-2 -(n+2)m \right)^2 }{2(n-2-nm)(1-m)\widetilde \beta}\left(\frac{1+\log s}{s} \right).
\end{equation}
Then 
\begin{equation}\label{h2ss}
\begin{aligned}
&h_{2,ss}+\left(\frac{2(n-2-nm)}{(1-m)}s + \frac{\widetilde \beta}{n-1}h- \frac{n-2-(n+2)m}{1-m} \right) h_{2,s}\\
=&\frac{1-2m}{1-m} \frac{\widetilde w_s^2}{\widetilde w}+\frac{(n-1) \left( n-2 -(n+2)m \right)^2 }{(1-m)^2\widetilde \beta} \cdot \frac{1}{s}
- \frac{n-2-(n+2)m}{1-m}\frac{h_1}{s}+\frac{a_3}{s^2}\\
&+\frac{(n-1) \left( n-2 -(n+2)m \right)^2 }{2(n-2-nm)(1-m)\widetilde \beta}\left[ \frac{(1-2\log s)}{s^3}+\left( \frac{\widetilde \beta}{n-1}h(s) -\frac{n-2-(n+2)m}{1-m}\right)\frac{\log s}{s^2} \right]\\
=&:\psi_3(s).
\end{aligned}
\end{equation}

\begin{lem}\label{lim s2h2s}
Let $n\ge3$, $0<m<\frac{n-2}{n}$, $m \neq \frac{n-2}{n+2}$ and $\tilde \alpha$, $\tilde \beta$ satisfy \eqref{alpha-beta-and-tilde-relation}. Then
\begin{equation}
\lim_{s\to \infty} s^2 h_{2,s}(s)=\frac{(1-m)a_2(\eta, \widetilde \beta)}{2(n-2-nm)\widetilde \beta},
\end{equation}
where $a_2(\eta, \widetilde \beta)$ is given by \eqref{a2-defn}
with $K(\eta, \widetilde \beta)$ given by \eqref{h1-infty-limit}.
\end{lem}
\begin{proof}
Let $\psi$ be as given by \eqref{psi-defn}.
By Theorem \ref{w-limit-thm}, Lemma \ref{hs}, and  Corollary \ref{h1},
\begin{equation}
\lim_{s\to\infty}s\psi_3(s)=\frac{a_2(\eta, \widetilde \beta)}{\widetilde \beta},
\end{equation}
where $a_2(\eta, \widetilde \beta)$ is given by \eqref{a2-defn} with $K(\eta, \widetilde \beta)$ given by \eqref{h1-infty-limit}. Then for any $0<\3<1$, there exists a constant $s_2>0$ such that
\begin{align}\label{h2-ineqn1}
\frac{(a_2(\eta, \widetilde \beta)/\4{\beta})-\3}{s}\le&h_{2,ss}+\left(\frac{2(n-2-nm)}{(1-m)}s + \frac{\widetilde \beta}{n-1}h- \frac{n-2-(n+2)m}{1-m} \right) h_{2,s}\notag\\
\le&\frac{(a_2(\eta, \widetilde \beta)/\4{\beta})+\3}{s}\quad\forall s\ge s_2.
\end{align}
Multiplying \eqref{h2-ineqn1} by $\psi$ and integrating over $(s_2,s)$,
\begin{align}\label{h2s-ineqn}
&\left(\frac{a_2(\eta, \widetilde \beta)}{\4{\beta}}-\3\right)\frac{\psi(s_2)h_{2,s}(s_2)+\int_{s_2}^2z^{-1}\psi(z)\,dz}{s^{-2}\psi(s)}
\le s^2h_{2,s}(s)\notag\\
&\qquad\qquad\qquad\qquad\qquad\qquad \le\left(\frac{a_2(\eta, \widetilde \beta)}{\4{\beta}}+\3\right)\frac{\psi(s_2)h_{2,s}(s_2)+\int_{s_2}^2z^{-1}\psi(z)\,dz}{s^{-2}\psi(s)}\quad\forall s\ge s_2.
\end{align}
Letting first $s\to\infty$ and then $\3\to 0$ in \eqref{h2s-ineqn}, by Lemma \ref{hs} and the l'Hospital rule,
\begin{align*}
&\lim_{s\to\infty}s^2h_{2,s}(s)\notag\\
=&\frac{a_2(\eta, \widetilde \beta)}{\4{\beta}}\lim_{s\to\infty}\frac{s^{-1}\psi(s)}{-2s^{-3}\psi (s)+s^{-2}\psi(s)\left(\frac{2(n-2-nm)}{1-m}s+\frac{\4{\beta}}{n-1}h(s)-\frac{n-2-(n+2)m}{1-m}\right)}\notag\\
=&\frac{(1-m)a_2(\eta, \widetilde \beta)}{2(n-2-nm)\widetilde \beta}
\end{align*}
and the lemma follows.
\end{proof}

\begin{cor}\label{lim s2-tail-small-at-infty-lem}
Let $n\ge3$, $0<m<\frac{n-2}{n}$, $m \neq \frac{n-2}{n+2}$ and $\tilde \alpha$, $\tilde \beta$ satisfy \eqref{alpha-beta-and-tilde-relation}. Then
\begin{equation}\label{h2-expansion}
h_2(s)=-\frac{(1-m)a_2(\eta, \widetilde \beta)}{2(n-2-nm)\widetilde \beta}\cdot\frac{1}{s}+o\left(s^{-1}\right)\quad\mbox{ as }s\to\infty,
\end{equation}
where $a_2(\eta, \widetilde \beta)$ is given by \eqref{a2-defn} with $K(\eta, \widetilde \beta)$ given by \eqref{h1-infty-limit}.
\end{cor}
\begin{proof}
Since $h_2(s)\to 0$ as $s\to\infty$, by Lemma \ref{lim s2h2s},
\begin{equation*}
h_s(s)=\int_{\infty}^sh_{2,s}(z)\,dz=\left(\frac{(1-m)a_2(\eta, \widetilde \beta)}{2(n-2-nm)\widetilde \beta}+o(1)\right)\int_{\infty}^sz^{-2}\,dz
=\left(\frac{(1-m)a_2(\eta, \widetilde \beta)}{2(n-2-nm)\widetilde \beta}+o(1)\right)\frac{1}{s}
\end{equation*}
as $s\to\infty$ and the corollary follows.
\end{proof}

Let
\begin{equation}\label{h3s d}
h_3(s)= h_{2}(s)+\frac{(1-m) a_2(\eta, \widetilde \beta)}{2(n-2-nm) \widetilde \beta s},
\end{equation}
where $a_2(\eta, \widetilde \beta)$ is given by \eqref{a2-defn} with $K(\eta, \widetilde \beta)$ given by \eqref{h1-infty-limit}.

\begin{cor}\label{h3}
Let $n\ge3$, $0<m<\frac{n-2}{n}$, $m \neq \frac{n-2}{n+2}$ and $\tilde \alpha$, $\tilde \beta$ satisfy \eqref{alpha-beta-and-tilde-relation}. Then
\begin{equation}\label{h3s}
\lim_{s\to \infty} s^2 h_{3, s}(s) =0.
\end{equation}
\end{cor}
\begin{proof}
By \eqref{h3s d} and Lemma \ref{lim s2h2s},
\begin{equation*}
\lim_{s\to\infty}s^2h_{3,s}(s)= \lim_{s\to\infty}s^2h_{2,s}(s)-\frac{(1-m) a_2(\eta, \widetilde \beta)}{2(n-2-nm)\widetilde{\beta}}=0
\end{equation*}
and the corollary follows.
\end{proof}

Since $\4{w}(s)=w(r)$, where $s=\log r$, by \eqref{alpha-beta-and-tilde-relation}, \eqref{w-defn}, \eqref{def h}, \eqref{def h1}, \eqref{h2-defn}, \eqref{h2-expansion}, Corollary \ref{h-limit6-lem}, we have the following the proposition.

\begin{prop}\label{g11-expansion-general}
Let $n\ge3$ and $0<m < \frac{n-2}{n}$. Let $a_1$ and $K_0$ be given by \eqref{a1-defn}  and \eqref{k0-defn} respectively. Let $g_{1,1}$ be the unique solution of \eqref{g-eqn} given by Theorem \ref{g-existence-thm}  with $\eta=1$, $\4{\beta}=1$ and $\4{\alpha}$ be given by \eqref{alpha-beta-and-tilde-relation}. Then the following holds.
\begin{equation}\label{class g}
\begin{aligned}
g_{1,1}(r)^{1-m}&=\frac{2(n-1)(n-2-nm)}{(1-m)r^\frac{n-2-nm}{m}}\left\{ \log r +\frac{(n-2-(n+2)m)}{2(n-2-nm)}\log (\log r)+K_0 \right.\\
&\left.+\frac{a_1}{\log r} +\frac{(n-2-(n+2)m)^2}{4(n-2-nm)^2}\cdot \frac{\log (\log r)}{\log r}+ o \left( \frac{1}{\log r} \right) \right \} \quad \text{ as } r\to \infty.
\end{aligned}
\end{equation}
\end{prop}

Recall that for any constants $\beta<0$ and $A>0$, we let $\alpha$ be as given by \eqref{ab} and $f_{\beta, A}$ be the unique solution of \eqref{ode} which satisfies \eqref{gro 0}, \eqref{gro inf} given by Theorem \ref{f-existence-uniqueness-thm}. 

\begin{lem}\label{f-transform-formula-lem}
Let $n\ge3$, $0<m<\frac{n-2}{n}$, $A_1>0$, $A_2>0$,  $\beta_1<0$, $\beta_2<0$, and  $\gamma_1$ be given by  \eqref{gamma1-defn}. Then
\begin{equation}\label{fA12}
f_{\beta_1, A_1}(r) =\left(A_2/A_1\right)^{\frac{2}{(1-m)\gamma_1}} f_{\beta_1, A_2}\left ((A_2/A_1)^{\frac{1}{\gamma_1}}r \right)\quad\forall r>0
\end{equation}
and 
\begin{equation}\label{fbe12}
f_{\beta_1, (\beta_2/\beta_1)^{\frac{1}{1-m}} A_1}(r) =(\beta_2/\beta_1)^{\frac{1}{1-m}}f_{\beta_2, A_1}(r)\quad\forall r>0.
\end{equation}
\end{lem}
\begin{proof}
Let 
\begin{equation*}
F_1(r)=\mu^{\frac{2}{1-m}}f_{\beta_1, A_2}(\mu r)\quad\forall\mu>0, r>0
\end{equation*}
and
\begin{equation*}
F_2(r)=(\beta_2/\beta_1)^{\frac{1}{1-m}}f_{\beta_2, A_1}(r)\quad\forall r>0,
\end{equation*}
where $\mu =(A_2/A_1)^{\frac{1}{\gamma_1}}$. Then both $F_1$ and $F_2$ satisfies \eqref{ode} with $\alpha=\frac{2\beta_1}{1-m}$, $\beta=\beta_1$. Moreover
\begin{equation*}
\lim_{r\to\infty}r^{\frac{n-2}{m}}F_1(r)=\mu^{\frac{2}{1-m}-\frac{n-2}{m}}\lim_{r\to\infty}(\mu r)^{\frac{n-2}{m}}f_{\beta_1, A_2}(\mu r)=A_2 \mu^{\frac{2}{1-m}-\frac{n-2}{m}}=A_1
\end{equation*}
and
$$
\lim_{r\to\infty}r^{\frac{n-2}{m}}F_2(r)=(\beta_2/\beta_1)^{\frac{1}{1-m}}\lim_{r\to\infty}r^{\frac{n-2}{m}} f_{\beta_2, A_1}(r)=(\beta_2/\beta_1)^{\frac{1}{1-m}}A_1.
$$
Hence by Theorem \eqref{f-existence-uniqueness-thm}, 
\begin{equation*}
F_1(r) = f_{\beta_1, A_1}(r)\quad\forall r>0
\end{equation*}
and
\begin{equation*}
F_2(r)= f_{\beta_1, (\beta_2/\beta_1)^{\frac{1}{1-m}}A_1}(r)\quad\forall r>0.
\end{equation*}
Thus \eqref{fA12} and \eqref{fbe12} follows.
\end{proof}

\begin{lem}\label{g-transform-formula-lem}
Let $n\ge3$ and $0<m<\frac{n-2}{n}$, and $\eta_1>0, \eta_2>0$, $\4{\beta}_1>0$, $\4{\beta}_2>0$. Then 
\begin{equation}\label{fA12a}
g_{\widetilde \beta_1, \eta_1}(r) = (\eta_1/ \eta_2) g_{\widetilde \beta_1, \eta_2}\left((\eta_1/ \eta_2)^{\frac{m(1-m)} {n-2-nm}} r\right)\quad\forall r\ge 0
\end{equation}
and
\begin{equation}\label{fbe12a}
g_{\4{\beta}_1, (\4{\beta}_2/\4{\beta}_1)^{\frac{1}{1-m}} \eta_1}(r) =(\widetilde \beta_2/ \widetilde \beta_1)^{\frac{1}{1-m}}g_{\widetilde \beta_2, \eta_1}(r)\quad\forall r\ge 0.
\end{equation}
\end{lem}
\begin{proof}
Let $\4{\alpha}_1$ be given by \eqref{alpha-beta-and-tilde-relation} with $\4{\alpha}=\4{\alpha}_1$ and $\4{\beta}=\4{\beta}_1$.
Let
\begin{equation*}
G_1(r):=(\eta_1/ \eta_2) g_{\widetilde \beta_1, \eta_2} \left((\eta_1/ \eta_2)^{\frac{m(1-m)} {n-2-nm}} r\right)\quad\forall r\ge 0
\end{equation*}
and 
\begin{equation*}
G_2(r):=(\widetilde \beta_2/ \widetilde \beta_1)^{\frac{1}{1-m}}g_{\widetilde \beta_2, \eta_1}(r)
\quad\forall r\ge 0.
\end{equation*}
Then $G_1$  satisfies \eqref{g-eqn} with $\eta=\eta_1$ and  $\4{\alpha}=\4{\alpha}_1$, $\4{\beta}=\4{\beta}_1$, and $G_2$  satisfies \eqref{g-eqn} with $\eta=(\widetilde \beta_2/ \widetilde \beta_1)^{\frac{1}{1-m}}\eta_1$ and  $\4{\alpha}=\4{\alpha}_1$, $\4{\beta}=\4{\beta}_1$. Hence by Theorem \ref{g-existence-thm}, 
\begin{equation*}
G_1(r) = g_{\widetilde \beta_1, \eta_1}(r)\quad\forall r\ge 0
\end{equation*}
and
\begin{equation*}
G_2(r)= g_{\4{\beta}_1, (\4{\beta}_2/\4{\beta}_1)^{\frac{1}{1-m}} \eta_1}(r)\quad\forall r\ge 0.
\end{equation*}
Thus \eqref{fA12a} and \eqref{fbe12a} follows.
\end{proof}

We are now ready for the proof of Theorem \ref{Thm HAB}.

\begin{proof}[Proof of Theorem \ref{Thm HAB}]
By Proposition \ref{g11-expansion-general},

\begin{align}\label{f-1-1-expand}
f_{-1,1}^{1-m}(r)&=r^{-\frac{n-2}{m}(1-m)} g^{1-m}_{1,1}(r^{-1})\notag\\
&=\frac{2(n-1)(n-2-nm)}{(1-m)r^2}\left\{\log r^{-1} +\frac{(n-2-(n+2)m)}{2(n-2-nm)}\log (\log r^{-1})+K_0 \right.\notag\\
&\left. +\frac{a_1}{\log r^{-1}}+\frac{(n-2-(n+2)m)^2}{4(n-2-nm)^2}\cdot \frac{\log (\log r^{-1})}{\log r^{-1}}+ o\left(\frac{1}{\log r^{-1}}\right) \right \} \quad \text{as } r\to 0.
\end{align}
Hence by \eqref{f-1-1-expand} and Lemma \ref{f-transform-formula-lem}, 

\begin{align}\label{f-lambda-expansion5}
&f_{\beta, A}(r)^{1-m}\notag\\
=&\left( A (-\beta)^{\frac{1}{1-m}} \right)^{-\frac{2}{\gamma_1}} 
f_{\beta, (-\beta)^{-\frac{1}{1-m}}}^{1-m}\left( \left( A (-\beta)^{\frac{1}{1-m}} \right)^{-\frac{1}{\gamma_1}} r \right)\notag\\
=&\left( A (-\beta)^{\frac{1}{1-m}} \right)^{-\frac{2}{\gamma_1}} (-\beta)^{-1} f_{-1, 1}^{1-m}\left( \left( A (-\beta)^{\frac{1}{1-m}} \right)^{-\frac{1}{\gamma_1}} r \right)\notag\\
=&\frac{2(n-1)(n-2-nm)}{(1-m)|\beta| r^2}\left\{\log r^{-1} +\frac{(n-2-(n+2)m)}{2(n-2-nm)}\log \left(\log \left( \left( A (-\beta)^{\frac{1}{1-m}} \right)^{\frac{1}{\gamma_1}} r^{-1} \right)\right)\right.\notag\\
&\qquad \left. +K_0+\frac{1}{\gamma_1} \log A +\frac{1}{\gamma_1 (1-m)} \log|\beta| +\frac{a_1}{\log r^{-1}} \right.\notag\\
&\qquad \left. +\frac{(n-2-(n+2)m)^2}{4(n-2-nm)^2}\cdot \frac{\log \left(\log \left( \left( A (-\beta)^{\frac{1}{1-m}} \right)^{\frac{1}{\gamma_1}} r^{-1}  \right)\right)}{\log \left( \left( A (-\beta)^{\frac{1}{1-m}} \right)^{\frac{1}{\gamma_1}} r^{-1} \right)}+ o((\log r^{-1})^{-1}) \right \}\,\,\mbox{ as }r\to 0.
\end{align}
Now
\begin{align}\label{log-log-expansion}
\log \left(\log \left( \left( A (-\beta)^{\frac{1}{1-m}} \right)^{\frac{1}{\gamma_1}} r^{-1} \right)\right)=&\log (\log r^{-1})+\log\left(1+\frac{\log (A (-\beta)^{\frac{1}{1-m}})}{\gamma_1\log r^{-1}}\right)\notag\\
=&\log (\log r^{-1})+\frac{\log (A (-\beta)^{\frac{1}{1-m}})}{\gamma_1\log r^{-1}}+o((\log r^{-1})^{-1})\quad\mbox{ as }r\to 0.
\end{align}
Hence by \eqref{f-lambda-expansion5} and \eqref{log-log-expansion} we get \eqref{f-lambda-A-expansion}. Putting $A=\lambda^{-\gamma_1}$ in \eqref{f-lambda-A-expansion} we get \eqref{f-lambda-expansion} and the theorem follows. 
\end{proof}

\begin{cor}\label{dif f}
Let $n\ge 3$, $0<m<\frac{n-2}{n}$ and $\lambda_1>\lambda_2>0$. Then there exist  constants $0<\delta_2<1$ and $c_4>c_3>0$ such that
\begin{equation}\label{self-similar-solns-difference}
c_3r^{-\frac{2}{1-m}}(\log r^{-1})^{\frac{m}{1-m}} \le f_{\lambda_2}(r)-f_{\lambda_1}(r)\le c_4r^{-\frac{2}{1-m}}(\log r^{-1})^{\frac{m}{1-m}} \quad\forall 0<r\le \delta_2.
\end{equation}
\end{cor}
\begin{proof}
By the mean value theorem for any $r>0$ there exists a  constant $\xi_1(r)$ between $f_{\lambda_1}^{1-m}(r)$ and $f_{\lambda_2}^{1-m}(r)$ such that
\begin{equation}\label{f-lambda-difference}
f_{\lambda_2}(r)-f_{\lambda_1}(r)=(f_{\lambda_2}^{1-m}(r))^{\frac{1}{1-m}}-(f_{\lambda_1}^{1-m}(r))^{\frac{1}{1-m}}=\frac{1}{1-m}(f_{\lambda_2}^{1-m}(r)-f_{\lambda_1}^{1-m}(r))\xi_1(r)^{\frac{m}{1-m}}.
\end{equation}
By \eqref{f-lambda-expansion} and \eqref{f-lambda-difference},
\begin{equation}\label{f-difference11}
f_{\lambda_2}(r)-f_{\lambda_1}(r)\approx\frac{2(n-1)(n-2-nm)}{(1-m)^2|\beta| r^2}\cdot\left(\frac{2(n-1)(n-2-nm)}{(1-m)|\beta|}\left(\frac{\log r^{-1}}{r^2}\right) \right)^{\frac{m}{1-m}}\log (\lambda_1/\lambda_2)
\end{equation}
as $r\to 0$.
Hence by \eqref{f-difference11}  there exist  constants $0<\delta_2<1$ and $c_4>c_3>0$  such that \eqref{self-similar-solns-difference} holds and the lemma follows.
\end{proof}

Since $g_{\widetilde \beta, \eta}(r)=r^{-\frac{n-2}{m}}f_{\beta,\eta}(r^{-1})$, where $f_{\beta,\eta}$ is the unique solution of \eqref{ode} which satisfies  \eqref{gro inf} with $A=\eta$, by Theorem \ref{Thm HAB} we have the following result.

\begin{cor}\label{g-expansion20}
Let $n\ge3$, $0<m<\frac{n-2}{n}$, $\eta>0$ and $\widetilde \alpha$, $\widetilde \beta$ be given by \eqref{alpha-beta-and-tilde-relation}. Let $K_0$ be given by \eqref{k0-defn} and $a_3(\eta,\4{\beta})$ be given by \eqref{a3-defn} with $A=\eta$. Then the following holds.

\begin{align*}
g_{\widetilde \beta, \eta}^{1-m}(r)&=\frac{2(n-1)(n-2-nm)}{(1-m)\widetilde \beta r^\frac{n-2-nm}{m}}\left\{ \log r +\frac{(n-2-(n+2)m)}{2(n-2-nm)}\log (\log r) \right. \notag\\
&+K_0+\frac{1} {\gamma_1} \log \eta+ \frac{m} {n-2-nm} \log \widetilde \beta\notag \\
&\left. +\frac{a_3(\eta,\4{\beta})}{\log r} +\frac{(n-2-(n+2)m)^2}{4(n-2-nm)^2}\cdot \frac{\log (\log r)}{\log r}+ o\left(\frac{1}{\log r}\right) \right \} \quad \text{ as } r\to \infty.
\end{align*}
\end{cor}

\begin{remark}
One can also prove Corollary \ref{g-expansion20} by using Proposition \ref{g11-expansion-general}, Lemma \ref{g-transform-formula-lem} and an argument similar to the proof of 
Theorem \ref{Thm HAB}.
\end{remark}

By Theorem \ref{w-limit-thm} and Corollary \ref{g-expansion20} and an argument similar to the proof of Corollary \ref{dif f}, we have the following result.

\begin{cor}\label{g-lambda-12-difference-cor}
Let $n\ge 3$, $0<m<\frac{n-2}{n}$, $\lambda>0$, $\lambda_1>\lambda_2>0$, and $\alpha$, $\beta$, $\widetilde \alpha$, $\widetilde \beta$ be given by \eqref{ab} and \eqref{tilde ab} respectively. Let $g_{\lambda_i}$, $i=1,2$, be as given by \eqref{def g lam} with $\lambda=\lambda_1,\lambda_2$. Then there exist  constants $R_1>1$, $R_2(\lambda)>0$ and $c_6>c_5>0$ such that
\begin{equation*}
c_5r^{\frac{2}{1-m}-\frac{n-2}{m}}(\log r)^{\frac{m}{1-m}} \le g_{\lambda_2}(r)-g_{\lambda_1}(r)\le c_6r^{\frac{2}{1-m}-\frac{n-2}{m}}(\log r)^{\frac{m}{1-m}} \quad\forall r\ge R_1.
\end{equation*}
and
\begin{equation*}
c_5|x|^{-\frac{n-2-nm}{m}}\log |x|\le g_{\lambda}^{1-m}(x)\le c_6|x|^{-\frac{n-2-nm}{m}}\log |x|\quad\forall |x|\ge R_2(\lambda).
\end{equation*}
\end{cor}

\section{Existence, uniqueness and asymptotic large time behaviour of singular solutions}\label{section-existence-uniqueness-asymptotic-behaviour}
\setcounter{equation}{0}
\setcounter{theorem}{0}

In this section, we will prove the existence, uniqueness, and asymptotic large time behaviour of singular solutions of \eqref{main-eq}. We first prove the $L^1$-contraction result with weight $|x|^{-\mu}$.

\begin{proof}[Proof of Theorem \ref{u-v-weight-mu-L1-contraction}]

We will use a modification of the proof of Lemma 4.1 of \cite{DS1} and Theorem 1.2 of \cite{HK} to prove the theorem. We choose $\eta\in C_0^\infty(\R^n)$ such that $0\leq \eta\leq1, $ $\eta=1$ for $|x|\leq 1,$ and $\eta=0$ for $|x|\geq 2.$ For any $R>2,$ and $0<\ve<1,$ let $\eta_{R}(x):=\eta(x/R)$, $\eta_\ve(x):=\eta(x/\ve)$, and $\eta_{\ve,R}(x)=\eta_R(x)-\eta_\ve(x).$ Then $|\D \eta_{\ve,R}|^2 +|\La \eta_{\ve,R}|\leq C\ve^{-2}$ for $\ve\leq |x|\leq 2\ve, $ and $|\D \eta_{\ve,R}|^2 +|\La \eta_{\ve,R}|\leq CR^{-2}$ for $R\leq |x|\leq 2R. $ Let $r_1$ and $\delta_1\in (0,1)$ be as given in Remark \ref{flambda12a} and Corollary \ref{dif f} respectively and $R_0=\max(2,r_1/\lambda_2)$.

Recall that by the Kato inequality (\cite{DK}, \cite{K}),
$$\Delta ( \phi_1- \phi_2 )_+ \ge \chi_{\{\phi_1 > \phi_2\}} \Delta ( \phi_1 - \phi_2 ) \quad \text{ in } \mathscr D'(\Omega),$$
where $\phi_1$, $\phi_2$ are $C^2$ functions on a domain $\Omega$. Since
$$
\frac{\partial}{\partial t}(u_1-u_2)_+ =\chi_{\{u>v\}}\frac{\partial}{\partial t}(u_1-u_2)=\chi_{\{u_1>u_2\}}\frac{n-1}{m} \Delta (u_1^m - u_2^m)\quad\mbox{ in }
\left(\R^n\setminus\{0\}\right)\times(0,\infty),
$$
by the Kato inequality,
\begin{equation*}
\left\{\begin{aligned}
&\frac{\partial}{\partial t} (u_1-u_2)_+ \leq \frac{n-1}{m} \Delta \left( u_1^m-u_2^m \right)_+ \quad\mbox{in $\mathscr D' \left(\left(\R^n\setminus\{0\}\right)\times(0,\infty)\right)$}\\
&\frac{\partial}{\partial t} (u_1-u_2)_- \leq \frac{n-1}{m} \Delta \left( u_1^m-u_2^m \right)_- \quad\mbox{in $\mathscr D' \left(\left(\R^n\setminus\{0\}\right)\times(0,\infty)\right)$}.
\end{aligned}\right.
\end{equation*}
Hence
\begin{equation*}
\frac{\partial}{\partial t} |u_1-u_2| \leq \frac{n-1}{m} \Delta |u_1^m-u_2^m| \quad\mbox{in $\mathscr D' \left(\left(\R^n\setminus\{0\}\right)\times(0,\infty)\right)$}.
\end{equation*}
Thus for any $t>0$, $0<\3\le\delta_1$, and $R\ge R_0$,
\begin{align}\label{differential-ineqn1}
&\frac{d}{dt}\int_{\re^n\setminus \{0\}} |u_1-u_2|(x,t) \eta_{\ve, R}(x) |x|^{-\mu}\, dx\notag\\
\le&\int_{\re^n\setminus \{0\}} |u_1^m-u_2^m|(x,t) \Delta \left(\eta_{\ve, R}(x) |x|^{-\mu}\right) \,dx\notag\\
=&\int_{\re^n\setminus \{0\}} |u_1^m-u_2^m|(x,t)(|x|^{-\mu} \Delta \eta_{\ve, R}(x)+2\nabla \eta_{\ve, R} \cdot \nabla |x|^{-\mu}+ \eta_{\ve, R} \Delta |x|^{-\mu})\, dx\notag\\
\le&\int_{\re^n\setminus \{0\}} |u_1^m-u_2^m|(x,t)(|x|^{-\mu} \Delta \eta_{\ve, R}(x)+2\nabla \eta_{\ve, R} \cdot \nabla |x|^{-\mu})\, dx
\end{align}
since $0< \mu\le \mu_1<n-2$ and $\La |x|^{-\mu}=\mu ( \mu -(n-2) )|x|^{-\mu-2}<0\quad\forall x\in\R^n\setminus\{0\}$.

Let $T_0>0$. By integrating \eqref{differential-ineqn1} over $(0,t)$, by \eqref{def Ulam}, \eqref{uvU} and Remark \ref{flambda12a}, for any $0<t<T_0$, $0<\3\le\delta_1$, and $R\ge R_0$,
\begin{equation}\label{weighted-L1-contraction-mu}
\begin{aligned}
&\int_{\R^n \setminus \{0\}}|u_1-u_2|(x,t) \eta_{\ve,R}(x) |x|^{-\mu}\,dx-\int_{\R^n \setminus \{0\}}|u_{0,1}(x)-u_{0,2}(x)|\eta_{\ve,R}(x) |x|^{-\mu}\,dx\\
\le& C R^{-2-\mu}\int_0^t\int_{B_{2R}\setminus B_R} U_{\ld_2}^m(x,s) \, dx\, ds+ C \ve^{-2-\mu} \int_0^t\int_{B_{2\ve}\setminus B_\ve} \left( U_{\ld_2}^m(x,s)- U_{\ld_1}^{m}(x,s) \right) \,dx\,ds\\
=&C R^{-2-\mu} \int_0^te^{-m\alpha s}\int_{B_{2R}\setminus B_R}f_{\lambda_2}^m(e^{-\beta s} x)\, dx\,ds\\
&\qquad+C\ve^{-2-\mu} \int_0^t e^{-m\alpha s}\int_{B_{2\ve}\setminus B_\ve}\left(f_{\lambda_2}^m(e^{-\beta s}x)-f_{\lambda_1}^m(e^{-\beta s}x)\right)\,dx\,ds\\
\le&C R^{-2-\mu} \int_0^te^{-m\alpha s}\int_{B_{2R}\setminus B_R}|e^{-\beta s} x|^{2-n} \, dx\,ds\\
&\qquad+C\ve^{-2-\mu} \int_0^t e^{-m\alpha s}\int_{B_{2\ve}\setminus B_\ve}\left(f_{\lambda_2}^m(e^{-\beta s}x)-f_{\lambda_1}^m(e^{-\beta s}x)\right)\,dx\,ds\\
\le& CR^{-\mu} +C\ve^{-2-\mu} \int_0^t e^{-m\alpha s}\int_{B_{2\ve}\setminus B_\ve}\left(f_{\lambda_2}^m(e^{-\beta s}x)-f_{\lambda_1}^m(e^{-\beta s}x)\right)\,dx\,ds\\
\le& CR^{-\mu} +I_1.
\end{aligned}
\end{equation}
By Corollary \ref{dif f} for sufficiently small $\3\in (0,\delta_1)$,
\begin{align}\label{near-0-integral-bd}
I_1\le&C\ve^{-2-\mu}\int_0^{\,t}\int_{B_{2\ve}\setminus B_\ve} |x|^{-\frac{2m}{1-m}} \left(\log ((e^{-\beta s}|x|)^{-1} \right)^{\frac{m}{1-m}-1}\,dx\,ds\notag\\
\le&C \ve^{n-\frac{2}{1-m}-\mu} \left(\log \3^{-1} \right)^{\frac{2m-1}{1-m}}.
\end{align}
Since $T_0$ is arbitrary, letting $\ve\to 0$ and $R\to\infty$ in \eqref{weighted-L1-contraction-mu}, by \eqref{mu-m-range} and \eqref{near-0-integral-bd}, \eqref{uu2} follows. By an argument similar to the proof of \eqref{uu2} we get \eqref{u+u2} and the theorem follows.
\end{proof}

We are ready to prove Theorem \ref{thm exi uni u}.

\begin{proof}[Proof of Theorem \ref{thm exi uni u}]
We first observe that uniqueness of solutions of \eqref{main-eq} follows directly from Theorem \ref{u-v-weight-mu-L1-contraction}. Since the proof of the existence of solutions of \eqref{main-eq} is similar to the proof of Theorem 1.3 of \cite{HK}, we will only sketch the proof here. By Theorem 2.2 of \cite{H2}
there exists a unique solution $u_R\in C(\overline{A_R}\times(0,\infty))\cap C^{\infty}(A_R\times(0,\infty))$ of
\begin{equation*}
\left\{\begin{aligned}
u_t=\La u^m\quad&\mbox{ in $A_R\times(0,\infty),$}\\
u =U_{\ld_1}\quad&\mbox{ in $\partial A_R \times(0,\infty),$}\\
u(\cdot,0)=u_0\quad&\mbox{ in $A_R,$}
\end{aligned}\right.
\end{equation*}
which satisfies \eqref{fde}, \eqref{utu} and \eqref{Ulam1u ent} in $A_R\times(0,\infty)$ and
$$
\|u_R(\cdot, t)- u_0\|_{L^1(A_R)}\to 0\quad\mbox{as $t\to0$}.
$$
Let $\{R_i\}_{i=1}^{\infty}\in (0,\infty)$ be a sequence such that $R_i\to\infty$ as $i\to\infty$. Since $u_{R_i}$ satisfies \eqref{Ulam1u ent} in $A_{R_i}\times(0,\infty)$ for any $i\in\Z^+$, the equation \eqref{fde} for the sequence $\{u_{R_i}\}_{i=1}^{\infty}$ is uniformly parabolic on every compact subset $K$ of $\R^n\setminus\{0\}$. Hence by the parabolic Schauder estimates \cite{LSU} the sequence $\{u_{R_i}\}_{i=1}^{\infty}$ is uniformly bounded in
$C^{2+\mu_0,1+(\mu_0/2)}(K)$ on every compact subset $K$ of $\R^n\setminus\{0\}$ for some constant $\mu_0\in (0,1)$. Then by the Ascoli Theorem and a diagonalization argument the sequence $\{u_{R_i}\}_{i=1}^{\infty}$ has a subsequence which we may assume without loss of generality to be the sequence itself that converges uniformly in $C^{2,1}(K)$ for any compact subset $K$ of $\R^n\setminus\{0\}$ to the solution $u$ of \eqref{main-eq} which satisfies \eqref{utu} and \eqref{Ulam1u ent} as $i\to\infty$.

Finally if $u_0$ is radially symmetric, then by uniqueness of solution of \eqref{main-eq} for any $t>0$, $u(x, t)$ is radially symmetric in $x\in\R^n\setminus\{0\}$ and the theorem follows.
\end{proof}

We next will prove several technical lemmas before proving the $L^1$-contraction result with weight $f_{\lambda}^{m\gamma}$.

\begin{lem}\label{rf'f}
Let $n\ge 3$, $0<m<\frac{n-2}{n}$, and $\alpha$, $\beta$, $\widetilde \alpha$, $\widetilde \beta$ be given by \eqref{ab} and \eqref{tilde ab} respectively. Let $f=f_1$ be the unique solution of \eqref{ode} that satisfies \eqref{gro inf} with $A=1$ and let $g$ be as given by \eqref{def g} with $f=f_1$. Then
\begin{equation}\label{r-g'/r-limit}
\lim_{r\to\infty}\frac{rg_r(r)}{g(r)}=-\frac{\widetilde \alpha}{\widetilde \beta}
\end{equation}
and
\begin{equation}\label{r-f'/r-limit}
\lim_{r\to 0}\frac{rf_r(r)}{f(r)}=-\frac{2}{1-m}.
\end{equation}
Hence there exist constants $r_3>0$, $r_3'>0$, $C_1>C_2>0$ and $C_1'>C_2'>0$ such that
\begin{equation}\label{rf'-f-equivalence-a}
-C_1rf_r(r) \le f(r)\le -C_2rf_r(r) \quad \forall 0<r\le r_3
\end{equation}
and
\begin{equation}\label{rg'-g-equivalence5}
-C_1'rg_r(r) \le g(r)\le -C_2'rg_r(r) \quad \forall r\ge r_3'.
\end{equation}
\end{lem}
\begin{proof}
Let $w$ be as given by \eqref{w-defn}. By Lemma \ref{c123} and Theorem \ref{w-limit-thm},
\begin{align*}
&rw_r=(1-m)w(r) \left( \frac{\widetilde \alpha}{\widetilde \beta} +\frac{rg_r(r)}{g(r)}\right)\quad\forall r>0\notag\\
\Rightarrow\quad&(1-m)\lim_{r\to\infty}\left(\frac{\widetilde \alpha}{\widetilde \beta} +\frac{rg_r(r)}{g(r)}\right)
=\underset{\substack{r\to\infty}}{\lim}\frac{rw_r(r)}{w(r)}=\frac{\underset{\substack{r\to\infty}}{\lim}rw_r(r)}{\underset{\substack{r\to\infty}}{\lim}w(r)}=0
\end{align*}
and \eqref{r-g'/r-limit} follows.
By \eqref{f-g-relation}, \eqref{alpha-beta-and-tilde-relation} and \eqref{r-g'/r-limit},
\begin{align}
&\frac{rf_r(r)}{f(r)}=-\frac{n-2}{m}-\frac{r^{-1}g_r (r^{-1})}{g(r^{-1})}\quad\forall r>0\label{f'-expression}\\
\Rightarrow\quad&\lim_{r\to 0}\frac{rf_r(r)}{f(r)}=-\frac{n-2}{m}-\lim_{r\to 0}\frac{r^{-1}g_r(r^{-1})}{g(r^{-1})}=-\frac{n-2}{m}+\frac{\widetilde \alpha}{\widetilde \beta}=-\frac{2}{1-m}\notag
\end{align}
and \eqref{r-f'/r-limit} follows. By \eqref{r-g'/r-limit} and \eqref{r-f'/r-limit} there exist constants $r_3>0$, $r_3'>0$, $C_2>C_1>0$ and $C_2'>C_1'>0$ such that \eqref{rf'-f-equivalence-a} and \eqref{rg'-g-equivalence5} hold and the lemma follows.
\end{proof}

\begin{lem}\label{rf'f-infty-limit}
Let $n\ge 3$, $0<m<\frac{n-2}{n}$ and let $f=f_1$ be the solution of \eqref{ode} which satisfies \eqref{gro inf} with $A=1$. Let $g$ be as given by \eqref{def g} with $f=f_1$. Then
\begin{equation}\label{r-g'/g-infty-limit-g-at-origin}
\lim_{r\to 0}rg_r(r)=0,\qquad g(0)=1,
\end{equation}
and
\begin{equation}\label{r-f'/r-infty-limit}
\lim_{r\to\infty}\frac{rf_r(r)}{f(r)}=-\frac{n-2}{m}.
\end{equation}
Hence there exist constants $r_4>0$, $r_4'>0$, and $C_3>C_4>0$ such that
\begin{equation}\label{rf'-f-equivalence}
-C_3rf_r(r) \le f(r)\le -C_4rf_r(r) \quad \forall r\ge r_4
\end{equation}
and
\begin{equation}\label{rg'-g-ratio-bd}
|rg_r(r)| \le g(r)\quad \forall 0\le r\le r_4'.
\end{equation}
\end{lem}
\begin{proof}
By \eqref{gro inf}, \eqref{f'-expression} and Theorem \ref{g-existence-thm},
\begin{align*}
&\lim_{r\to 0}rg_r(r)=0\quad\mbox{ and }\quad g(0)=1\notag\\
\Rightarrow\quad&\lim_{r\to 0}\frac{rg_r(r)}{g(r)}=0\\
\Rightarrow\quad&\lim_{r\to\infty}\frac{rf_r(r)}{f(r)}=-\frac{n-2}{m}-\lim_{r\to\infty}\frac{r^{-1}g_r(r^{-1})}{g(r^{-1})}=-\frac{n-2}{m}
\end{align*}
and \eqref{r-g'/g-infty-limit-g-at-origin}, \eqref{r-f'/r-infty-limit}, follows. By \eqref{r-g'/g-infty-limit-g-at-origin} and \eqref{r-f'/r-infty-limit} there exist constants $r_4>0$, $r_4'>0$, and $C_3>C_4>0$ such that \eqref{rf'-f-equivalence} and \eqref{rg'-g-ratio-bd} hold and the lemma follows.
\end{proof}

By \eqref{f-lambda-defn}, \eqref{def g lam}, Lemma \ref{rf'f} and Lemma \ref{rf'f-infty-limit}, we have the following three corollaries.

\begin{cor}\label{rf'f-lambda-cor}
Let $n\ge 3$, $0<m<\frac{n-2}{n}$, $\lambda>0$, $f=f_{\lambda}$ be as given by \eqref{f-lambda-defn} and $r_3>0$, $C_1>C_2>0$ be as given by Lemma \ref{rf'f}. Then
\begin{equation*}
\lim_{r\to 0}\frac{rf_{\lambda,r}(r)}{f_{\lambda}(r)}=-\frac{2}{1-m}
\end{equation*}
and
\begin{equation*}
-C_1rf_{\lambda,r}(r) \le f_{\lambda}(r)\le -C_2rf_{\lambda,r}(r) \quad \forall 0<r\le r_3/\lambda.
\end{equation*}
\end{cor}

\begin{cor}\label{rf'f-lambda-infty-cor}
Let $n\ge 3$, $0<m<\frac{n-2}{n}$, $\lambda>0$, $f=f_{\lambda}$ be as given by \eqref{f-lambda-defn} and $r_4>0$, $C_3>C_4>0$ be as given by Lemma \ref{rf'f-infty-limit}. Then,
\begin{equation*}
\lim_{r\to\infty}\frac{rf_{\lambda,r}(r)}{f_{\lambda}(r)}=-\frac{n-2}{m}
\end{equation*}
and
\begin{equation*}
-C_3rf_{\lambda,r}(r) \le f_{\lambda}(r)\le -C_4rf_{\lambda,r}(r) \quad \forall r\ge r_4/\lambda.
\end{equation*}
\end{cor}

\begin{cor}\label{rg-lambda'-g-lambda-ratio-cor}
Let $n\ge 3$, $0<m<\frac{n-2}{n}$, $\lambda>0$, and $\alpha$, $\beta$, $\widetilde \alpha$, $\widetilde \beta$ be given by \eqref{ab} and \eqref{tilde ab} respectively. Let $f=f_1$ be the unique solution of \eqref{ode} that satisfies \eqref{gro inf} with $A=1$ and let $g_{\lambda}$ be as given by \eqref{def g lam}. Let $r_3'>0$ and $r_4'>0$ be as given by Lemma \ref{rf'f}
and Lemma \ref{rf'f-infty-limit} respectively. Then
\begin{equation}\label{r-g-lambda'/r-limit}
\lim_{r\to\infty}\frac{rg_{\lambda,r}(r)}{g_{\lambda}(r)}=-\frac{\widetilde \alpha}{\widetilde \beta}
\end{equation}
\begin{equation}\label{rg-lambda'-g-lambda-equivalence5}
-C_1'rg_{\lambda,r}(r) \le g_{\lambda}(r)\le -C_2'rg_{\lambda,r}(r) \quad \forall r\ge \lambda r_3'
\end{equation}
and
\begin{equation}\label{g-lambdaderivative-g-lambda-ratio-bd3}
|rg_{\lambda,r}(r)| \le g_{\lambda}(r)\quad \forall 0\le r\le\lambda r_4'.
\end{equation}
\end{cor}
\begin{proof}
By \eqref{f-lambda-defn} and \eqref{def g lam},
\begin{align}
&g_{\lambda}(r)=r^{-\frac{n-2}{m}}\lambda^{\frac{2}{1-m}}f_1(\lambda r^{-1})=\lambda^{\frac{2}{1-m}-\frac{n-2}{m}}g_1(r/\lambda)\label{g-lambda-g-1-relation}\\
\Rightarrow\quad&rg_{\lambda,r}(r)=\lambda^{\frac{2}{1-m}-\frac{n-2}{m}}(r/\lambda)g_{1,r}(r/\lambda)\quad\forall r\ge 0.\label{g-lambda-derivative-g-lambda-bd}
\end{align}
By \eqref{r-g'/r-limit}, \eqref{rg'-g-equivalence5}, \eqref{rg'-g-ratio-bd}, \eqref{g-lambda-g-1-relation}, and \eqref{g-lambda-derivative-g-lambda-bd}, we get \eqref{r-g-lambda'/r-limit}, \eqref{rg-lambda'-g-lambda-equivalence5}, and \eqref{g-lambdaderivative-g-lambda-ratio-bd3}, and
the corollary follows.
\end{proof}

We will now prove Theorem \ref{u-v-L1-contraction2}.

\begin{proof}[Proof of Theorem \ref{u-v-L1-contraction2}]
We will use a modification of the proof of Lemma 4.1 of \cite{DS1} and Theorem 1.2 of \cite{HK} to prove the theorem. Let $r_1$,  $r_2(\lambda)$, $r_3$, and $r_4$ be given by Remark \ref{flambda12a},  Lemma  \ref{rf'f} and Lemma \ref{rf'f-infty-limit}, respectively. Let  $T_0>0$,
\begin{equation*}
\3_0=\frac{\min (r_2(\lambda_1),r_2(\lambda_3), r_3,1)}{e^{2|\beta|T_0}\max (\lambda_3,2)}\quad\mbox{ and }\quad R_0=\max\left(2,\frac{\max (r_1,r_4)}{\min (\lambda_2,\lambda_3)}\right)
\end{equation*}
and
\begin{equation*}
h(x)=f_{\lambda_3}^m(x)\quad\forall x\in\R^n\setminus\{0\}.
\end{equation*}
Since $\frac{n-2}{n+2}\le m<\frac{n-2}{n} $ implies $0<\gamma_2\le 1$, by \eqref{ode} and \eqref{f dec}, 
\begin{equation}\label{laplace-f-m-gamma-neg}
\Delta h^{\gamma_2}=\gamma_2(\gamma_2-1) h^{{\gamma_2}-2}(h_r)^2+\gamma_2 h^{{\gamma_2}-1}\Delta h<0
\quad\mbox{ in }\R^n\setminus\{0\}.
\end{equation}
Let $q(x,t)=|u_1(x,t)-u_2(x,t)|$ and $\eta_{\ve, R}$ be as in the proof of Theorem \ref{u-v-weight-mu-L1-contraction}. By Kato's inequality (\cite{DK}, \cite{K}), 
\begin{equation}\label{eq-q-diff-rescaled-new}
q_t \leq \frac{n-1}{m} \Delta \left( a(x,t)q \right) \quad\mbox{in $\mathscr D '((\R^n\setminus\{0\})\times(0,\infty))$},
\end{equation}
where 
\begin{equation}\label{eq-a-lower-bd-new}
\begin{aligned}
m U_{\ld_2}^{m-1}(x,t)\leq a(x,t):=\int_0^1 \frac{m \, ds }{ \left\{s u_1+(1-s)u_2\right\}^{1-m}} \leq m U_{\ld_1}^{m-1}(x,t) \quad\forall x\in\R^n\setminus\{0\}, t>0.
\end{aligned}
\end{equation}
Then by \eqref{def Ulam}, \eqref{uvU}, \eqref{laplace-f-m-gamma-neg}, \eqref{eq-q-diff-rescaled-new}, \eqref{eq-a-lower-bd-new}, Remark \ref{flambda12a}, Corollary \ref{dif f}, Corollary \ref{rf'f-lambda-cor}, and Corollary \ref{rf'f-lambda-infty-cor}, for any $0<\3\le\3_0$, $R\ge R_0$ and $0<t<T_0$,
\begin{align}
&\frac{d}{dt}\int_{\R^n \setminus \{0\}}q(x,t)\eta_{\ve,R}(x) f_{\lambda_3}^{m \gamma_2} (x)\, dx\notag\\
\le&\frac{n-1}{m} \int_{\R^n \setminus \{0\}}a(x,t) q(x,t) \left( f_{\lambda_3}^{m \gamma_2} \Delta \eta_{\ve, R}+2\nabla f_{\lambda_3}^{m \gamma_2} \cdot \nabla \eta_{\ve, R} + \eta_{\ve, R} \Delta f_{\lambda_3}^{m \gamma_2}\right) \, dx\notag\\
\le&\frac{n-1}{m} \int_{\R^n \setminus \{0\}}a(x,t) q(x,t) \left[ f_{\lambda_3}^{m \gamma_2} \Delta \eta_{\ve, R}+2\nabla f_{\lambda_3}^{m \gamma_2} \cdot \nabla \eta_{\ve, R}
\right ] \, dx\notag\\
\le&C R^{-2} \int_{B_{2R}\setminus B_R}U_{\lambda_1}^{m-1}(x,t)U_{\lambda_2}(x,t)f_{\lambda_3}^{m \gamma_2} (x)\, dx\notag\\
&\qquad + C \ve^{-2}\int_{B_{2\3}\setminus B_{\3}}U_{\lambda_1}^{m-1}(x,t)q(x,t)f_{\lambda_3}^{m \gamma_2} (x)\, dx\label{integral-t-derivative12}\\
\le&C R^{-2} \int_{B_{2R}\setminus B_R}U_{\lambda_1}^{m-1}(x,t)U_{\lambda_2}(x,t)f_{\lambda_3}^{m \gamma_2} (x)\, dx\notag\\
&\qquad + C \ve^{-2}\int_{B_{2\3}\setminus B_{\3}}U_{\lambda_1}^{m-1}(x,t)(U_{\lambda_2}(x,t)-U_{\lambda_1}(x,t))f_{\lambda_3}^{m \gamma_2} (x)\, dx\notag\\
\le&C R^{-2}e^{-m\alpha t}\int_{B_{2R}\setminus B_R}f_{\lambda_1}^{m-1}(e^{-\beta t} x)f_{\lambda_2}(e^{-\beta t} x)f_{\lambda_3}^{m \gamma_2} (x)\, dx\notag\\
&\qquad + C \ve^{-2} e^{-m\alpha t}\int_{B_{2\3}\setminus B_{\3}}f_{\lambda_1}^{m-1}(e^{-\beta t} x)(f_{\lambda_2}(e^{-\beta t}x)-f_{\lambda_1}(e^{-\beta t}x))f_{\lambda_3}^{m \gamma_2}(x)\, dx.\notag\\
\le&C R^{n-2}e^{-m\alpha t}(e^{-\beta t}R)^{2-n}R^{-(n-2)\gamma_2}\notag\\
&\qquad + C \ve^{-2} e^{-m\alpha t}\int_{B_{2\3}\setminus B_{\3}}\left(\frac{e^{-2\beta t}\ve^2}{\log[(e^{- \beta t}\ve)^{-1}]} \right)\ve^{-\frac{2}{1-m}}\left(\log [(e^{- \beta t}\ve)^{-1}] \right)^{\frac{m}{1-m}}\left(\frac{\log\ve^{-1}}{\ve^2} \right)^{\frac{m\gamma_2}{1-m}}\, dx\notag\\
\le&CR^{-(n-2) \gamma_2}+C\left(\log\3^{-1}\right)^{\frac{n-4}{2}}.\label{integral-t-derivative}
\end{align}
Integrating \eqref{integral-t-derivative} over $(0,t)$, 
\begin{align}\label{q-f-lambda-l1-ineqn}
&\int_{\R^n\setminus\{0\}} q(x,t) \eta_{\ve,R} (x)f_{\lambda_3}^{m \gamma_2}(x)\, dx- \int_{\R^n\setminus\{0\}} q(x,0) \eta_{\ve,R} (x)f_{\lambda_3}^{m \gamma_2}(x)\, dx\notag\\
\le& C\left( R^{-(n-2)\gamma_2 } +\left(\log \ve^{-1} \right)^{\frac{n-4}{2}}\right)t \le CT_0
\quad\forall 0<\3\le\3_0, R\ge R_0, 0<t<T_0.
\end{align}
When $n=3$, letting $R\to\infty$ and $\3\to 0$ in \eqref{q-f-lambda-l1-ineqn}, we get that \eqref{uu0a} holds for $n=3$ and any $0<t<T_0$. Since $T_0$ is arbitrary, \eqref{uu0a} holds for $n=3$.

When $n=4$, letting $R\to\infty$ and $\3\to 0$ in \eqref{q-f-lambda-l1-ineqn}, we get
\begin{align}\label{q-f-lambda-l1-ineqn2}
&\int_{\R^n\setminus\{0\}} q(x,t)f_{\lambda_3}^{m \gamma_2}(x)\, dx
\le\int_{\R^n\setminus\{0\}} q(x,0) f_{\lambda_3}^{m \gamma_2}(x)\, dx+CT_0\le C'\notag\\
\Rightarrow\quad&\int_0^{T_0}\int_{\R^n\setminus\{0\}} q(x,t)f_{\lambda_3}^{m \gamma_2}(x)\, dx\,ds
\le C'T_0<\infty.
\end{align}
Then by \eqref{integral-t-derivative12} and an argument similar to the above one,
\begin{align}\label{q-f-lambda-l1-ineqn3}
&\int_{\R^n \setminus \{0\}}q(x,t)\eta_{\ve,R}(x) f_{\lambda_3}^{m \gamma_2} (x)\, dx\notag\\
\le&\int_{\R^n \setminus \{0\}}q(x,0)\eta_{\ve,R}(x) f_{\lambda_3}^{m \gamma_2} (x)\, dx
+CR^{-(n-2)\gamma_2 } +\frac{C}{\log\3^{-1}}\int_0^t\int_{\R^n\setminus\{0\}} q(x,s)f_{\lambda_3}^{m \gamma_2}(x)\, dx\,dt
\end{align}
holds for any $0<\3\le\3_0$, $R\ge R_0$ and $0<t<T_0$. Letting $R\to\infty$ and $\3\to 0$ in \eqref{q-f-lambda-l1-ineqn3}, by \eqref{q-f-lambda-l1-ineqn2}, we get that \eqref{uu0a} holds for $n=4$ and any $0<t<T_0$. Since $T_0$ is arbitrary, \eqref{uu0a} holds for $n=4$  and the theorem follows.

\end{proof}

\begin{cor}\label{L1 con cor}
Let $n= 3, 4$, $\frac{n-2}{n+2}\le m<\frac{n-2}{n}$,  and $\alpha$, $\beta$, $\gamma_2$ be as given by \eqref{ab} and \eqref{gamma2-defn} respectively. Let $\lambda_1> \lambda_2>0$, $\lambda_3>0$, and $f_{\lambda_i}$ be as given by \eqref{f-lambda-defn} with $\lambda=\lambda_1,\lambda_2, \lambda_3$. Let $u_{0,1}$, $u_{0,2}$ satisfy \eqref{initial_condition_lower-upper-bd}, \eqref{u0-v0-weighted-l1}, and $u_1$, $u_2$ be the solutions of
\eqref{main-eq} with intitial values $u_{0.1}$, $u_{0.1}$, respectively, which satisfy \eqref{uvU}. Let $\4{u}_1$, $\4{u}_2$ be given by \eqref{rescaled-soln-defn} with $u=u_1, u_2,$ respectively. Then 
\begin{equation}\label{uu0b}
\int_{\R^n\setminus\{0\}} |\widetilde u_1- \widetilde u_2|(x,t) f_{\lambda_3}^{m \gamma_2} (x)\, dx 
\le \int_{\R^n\setminus\{0\}} |\4{u}_{0,1}-\4{u}_{0,2}|(x) f_{e^{-\beta t}\lambda_3}^{m \gamma_2}(x) \, dx \quad \forall t>0.
\end{equation}
\end{cor}
\begin{proof}
We first observe that by \eqref{ab} and \eqref{gamma2-defn}, 
\begin{equation*}
\alpha - n\beta= -\frac{2\beta m\gamma_2}{1-m}.
\end{equation*}
Hence by \eqref{rescaled-soln-defn}, \eqref{f-lambda-defn}, Theorem \ref{u-v-L1-contraction2} and Remark \ref{flambda12},
\begin{align*}
\int_{\R^n\setminus\{0\}} |\widetilde u_1- \widetilde u_2|(x,t) f_{\lambda_3}^{m \gamma_2}(x)\, dx=&e^{\alpha t}\int_{\R^n\setminus\{0\}} |u_1-u_2|(e^{\beta t} x,t) f_{\lambda_3}^{m \gamma_2}(x)\, dx\notag\\
=&e^{(\alpha - n\beta)t}\int_{\R^n\setminus\{0\}}|u_1-u_2|(y,t) f_{\lambda_3}^{m \gamma_2}(e^{-\beta t}y)\, dy\notag\\
\le&\int_{\R^n\setminus\{0\}} |u_1-u_2|(y,t)\left(e^{-\frac{2\beta t}{1-m}} f_{\lambda_3}(e^{-\beta t}y)\right)^{m \gamma_2}\, dy\notag\\
=&\int_{\R^n\setminus\{0\}} |u_1-u_2|(y,t) f_{e^{-\beta t}\lambda_3}^{m \gamma_2}(y)\,dy\notag\\
\le&\int_{\R^n\setminus\{0\}} | u_{0,1}(y)- u_{0,2}(y)| f_{e^{-\beta t} \lambda_3}^{m \gamma_2}(y)\,dy\quad\forall t>0
\end{align*}
and \eqref{uu0b} follows. 
\end{proof}

\begin{lem}\label{f-lambda-lambda-infty-limit}
Let $n\ge 3$, $0<m<\frac{n-2}{n}$, $\lambda>0$ and $f=f_{\lambda}$ be as given by \eqref{f-lambda-defn}. Then $f_{\lambda}$ converges uniformly to zero on $[a,\infty)$ as $\lambda\to\infty$ for any constant $a>0$.
\end{lem}
\begin{proof}
Let $a>0$. Since $f_1$ satisfies \eqref{gro inf} with $A=1$, $|\lambda x|^{\frac{n-2}{m}}f_1(\lambda x)$ converges  to 1  as $\lambda\to\infty$ for any  $x\in\R^n\setminus\{0\}$. Hence by \eqref{f-lambda-defn},
\begin{equation*}
f_{\lambda}(x)=\lambda^{-\frac{n-2-nm}{m(1-m)}}|x|^{-\frac{n-2}{m}}|\lambda x|^{\frac{n-2}{m}}f_1(\lambda x)\to 0\quad\mbox{ uniformly  on }[a,\infty)\quad\mbox{ as }\lambda\to\infty
\end{equation*} 
and the lemma follows.
\end{proof}

We are now ready to prove Theorem \ref{thm assm}.

\begin{proof}[Proof of Theorem \ref{thm assm}]
Note that by \eqref{Ulam1u ent}, $\4{u}$ satisfies
\begin{equation}\label{lower-upper-bd-rescaled-soln2}
f_{\lambda_1} (x)\le \widetilde{u}(x,t)\le f_{\lambda_2}(x) \quad \forall x\in \re^n \setminus \{0\}, t>0.
\end{equation}
Hence  the equation \eqref{main-eq-resc} for $\4{u}$ is uniformly parabolic on $A_{2R}\times (1/2, \infty)$ for any $R>0$. By the Schauder estimates for uniformly parabolic equation \cite{LSU} $\widetilde u$ is uniformly bounded in $C^{2+\mu,1+(\mu/2)}(A_R\times (1, \infty))$  for some constant $0<\mu<1$ and any $R>0$. Let $\{t_i\}_{i=1}^{\infty}\subset (1,\infty)$ be such that $t_i \to \infty$ as $i \to \infty$ and $\4{u}_i(x,t)=\4{u}(x,t_i+t)$. Then by the Ascoli Theorem and a diagonalization argument the sequence $\{\widetilde u_i\}_{i=1}^{\infty}$ has subsequence which we may assume without loss of generality to be the sequence itself that  converges uniformly in $C^{2,1}(K)$ for any compact subset $K$ of $(\re^n\setminus\{0\})\times [0,\infty)$ to some function $\4{u}_{\infty}$ which satisfies \eqref{fde} in $(\re^n\setminus\{0\})\times [0,\infty)$ as $i \to \infty$.

Let $v_0(x)=\4{u}_{\infty}(x,0)$.
By Corollary \ref{L1 con cor},
\begin{equation}\label{ineqn100}
\int_{\R^n\setminus\{0\}} |\widetilde u(x,t)- f_{\lambda_0}(x)| f_{\lambda_3}^{m \gamma_2} (x)\, dx 
\le \int_{\R^n\setminus\{0\}} |\4{u}_0-f_{\lambda_0}(x)| f_{e^{-\beta t}\lambda_3}^{m \gamma_2}(x) \, dx \quad \forall t>0.
\end{equation}
Putting $t=t_i$ and letting $i\to\infty$ in \eqref{ineqn100},
by \eqref{u0-f0-weighted-l1}, Remark \ref{flambda12}, Lemma \ref{f-lambda-lambda-infty-limit} and the Lebesgue dominated convergence theorem,

\begin{align*}
\int_{\re^n\setminus\{0\}}|v_0(x)-f_{\lambda_0}(x)| f_{\lambda_3}^{m \gamma_2}(x) \, dx=0\quad\Rightarrow\quad v_0(x)\equiv f_{\lambda_0}(x)\quad\forall x\in \re^n\setminus\{0\}
\end{align*}
Since $t_i$ is arbitrary, $\widetilde u( \cdot, t)$ converges to $f_{\lambda_0}$ uniformly in $C^2(K)$ on every compact subset $K$ of $\re^n \setminus \{0\}$ as $t \to \infty$.
Letting $t\to\infty$ in \eqref{ineqn100}, by Remark \ref{flambda12}, Lemma \ref{f-lambda-lambda-infty-limit} and the Lebesgue dominated convergence theorem,
\begin{equation*}
\lim_{t\to\infty}\int_{\re^n\setminus\{0\}}|\widetilde {u}(x,t)-f_{\lambda_0}(x)| f_{\lambda_3}^{m \gamma_2}(x) \, dx=0.
\end{equation*}
and \eqref{u-tilde-l1-convergence} follows.
\end{proof}

\section{Asymptotic large time behaviour of radially symmetric singular solutions}
\label{section-asymptotic-behaviour-radially-symmetric-singular-soln}
\setcounter{equation}{0}
\setcounter{theorem}{0}

In this section we will study the asymptotic large time behaviour of radially symmetric solution $u$ of \eqref{main-eq} when the initial value $u_0$ is a radially symmetric function in $\R^n\setminus\{0\}$. We will use the inversion formula \eqref{u-bar-defn} to convert the problem into the study of the asymptotic large time behaviour of
the inversion function $\2{u}$.

We start by first proving a weighted $L^1$-contraction result for the inversion problem \eqref{Inversed eq}.

\begin{prop}\label{u12-bar-g-lambda-weight-l1-prop} 
Let $n$, $m$ satisfy \eqref{m-n-relation2} and  $\widetilde \alpha$, $\widetilde \beta$, $\gamma_3$ be as given by \eqref{tilde ab} and \eqref{gamma3-defn} respectively.
Let $\lambda_1> \lambda_2>0$, $\lambda_3>0$, and $g_{\lambda_i}$, $\overline U_{\lambda_i}$, $i=1,2,3$, be as given by \eqref{def g lam} and \eqref{def Ulam inv} respectively with $\lambda=\lambda_1,\lambda_2, \lambda_3$. Let $\overline u_{0,1}$, $\overline u_{0,2}$ satisfy
\begin{equation}\label{initial_condition_lower-upper-bd inv}
g_{\lambda_1} (x)\le \overline u_{0,i}(x)\le g_{\lambda_2}(x) \quad \text{ in } \re^n \setminus \{0\}\quad\forall i=1,2
\end{equation}
and
$$\overline u_{0,1}- \overline u_{0,2} \in L^1 \left (g_{\ld_3}^{m\gamma_3}; \re^n \setminus \{0\} \right).$$
Let $\overline u_1$, $\overline u_2$ be the solutions of \eqref{Inversed eq} with initial values $\overline u_{0,1}$, $\overline u_{0,2}$, respectively which satisfy
\begin{equation}\label{uvU inv1}
\overline U_{\lambda_1} \le \overline u_i\le \overline U_{\lambda_2} \quad \text{ in } (\re^n \setminus \{0\}) \times (0, \infty)\quad\forall i=1,2.
\end{equation}
Then
\begin{equation}\label{inv uu0}
\int_{\re^n \setminus \{0\}}|\overline u_1-\overline u_2|(x,t) g_{\ld_3}^{m \gamma_3}(x) \, dx 
\le \int_{\re^n \setminus \{0\}}|\overline u_{0,1}-\overline u_{0,2}|(x) g_{\ld_3}^{m \gamma_3}(x) \, dx \quad \forall t>0.
\end{equation}
\end{prop}
\begin{proof}
We will use a modification of the proof of Theorem 1.2 of \cite{HK} to prove this theorem.
Let $\eta$ and $\eta_{\3,R}$ be as in the proof of Theorem \ref{u-v-weight-mu-L1-contraction}
and let $R_1>0$, $R_2(\lambda)>0$, $r_3'>0$, $r_4'>0$ be as given by Corollary \ref{g-lambda-12-difference-cor}, Lemma \ref{rf'f}
and Lemma \ref{rf'f-infty-limit} respectively. Let $h_1(x)=g^m_{\lambda_3}(x)$. Since by \eqref{m-n-relation2} and \eqref{gamma3-defn} $0<\gamma_3\le 1$,
by \eqref{alpha-beta-and-tilde-relation}, \eqref{m-n-relation2}, \eqref{gamma3-defn},  \eqref{g-monotone-expression} and \eqref{gm-superharmonic},
\begin{align}\label{del rg n}
&\Delta \left(|x|^{n+2-\frac{n-2}{m}} g^{m \gamma_3}_{\lambda_3} (x)\right)\notag\\
=&h_1^{\gamma_3}(x)\La |x|^{n+2-\frac{n-2}{m}} + 2 \nabla |x|^{n+2-\frac{n-2}{m}} \cdot \nabla h_1^{\gamma_3}(x) + |x|^{n+2-\frac{n-2}{m}} \Delta h_1^{\gamma_3}(x)\notag\\
=&\left (\frac{n-2}{m}-(n+2)\right)\left(\frac{n-2}{m}-2n \right) |x|^{n-\frac{n-2}{m}}h_1^{\gamma_3}(x)
-2 \left(\frac{n-2}{m}-(n+2)\right) \gamma_3 |x|^{n+1-\frac{n-2}{m}}h_1^{\gamma_3-1}h_{1,r}\notag\\
&\qquad +\left(\gamma_3(\gamma_3-1) h_1^{\gamma_3-2}|\nabla h_1|^2+\gamma_3 h_1^{\gamma_3-1}\Delta h_1\right) |x|^{n+2-\frac{n-2}{m}}\notag\\
\le&\left(\frac{n-2}{m}-(n+2)\right)|x|^{n-\frac{n-2}{m}}h_1^{\gamma_3-1}(x) \left( \left(\frac{n-2}{m}-2n \right)h_1-2\gamma_3 r h_{1,r} \right)\notag\\
=&\left (\frac{n-2}{m}-(n+2)\right)|x|^{n-\frac{n-2}{m}}h_1^{\gamma_3-1}(x) g_{\lambda_3}^{m-1}(x) \left(\left(\frac{n-2}{m}-2n \right) g_{\lambda_3}(x)-2m\gamma_3 r g_{\lambda_3,r}(x)\right)\notag\\
<&\left (\frac{n-2}{m}-(n+2)\right)|x|^{n-\frac{n-2}{m}}h_1^{\gamma_3-1}(x) g_{\lambda_3}^{m}(x) \left(\left(\frac{n-2}{m}-2n \right)+2m\gamma_3 \left(\frac{\widetilde \alpha}{\widetilde \beta}\right) \right),\notag\\
=&-\left (\frac{n-2}{m}-(n+2)\right)\frac{(n-2-nm)}{m(1-m)}|x|^{n-\frac{n-2}{m}} h_1^{\gamma_3-1}(x) g_{\lambda_3}^{m}(x) ,\notag\\
<&0\quad\mbox{ in }\R^n\setminus\{0\},
\end{align}
where $r=|x|$.

Let $q(x,t)=|\overline u_1(x,t)-\overline u_2(x,t)|$. By Kato's inequality (\cite{DK}, \cite{K}),
\begin{equation}\label{eq-q-diff-rescaled-newn}
q_t \leq \frac{n-1}{m} |x|^{n+2- \frac{n-2}{m}} \Delta \left( a(x,t)q \right) \quad\mbox{in $\mathscr D '((\R^n\setminus\{0\})\times(0,\infty))$},
\end{equation}
where
\begin{equation}\label{eq-a-lower-bd-newn}
\begin{aligned}
m \overline U_{\ld_2}^{m-1}(x,t)\leq a(x,t):=\int_0^1 \frac{m \, ds }{ \left\{s \overline u_1+(1-s)\overline u_2\right\}^{1-m}} \leq m \overline U_{\ld_1}^{m-1}(x,t) \quad\forall x\in\R^n\setminus\{0\}, t>0.
\end{aligned}
\end{equation}

Let $T_0>0$,
\begin{equation*}
\3_1=\min\left(1/2,\lambda_3 r_4'\right)\quad\mbox{ and }\quad R_3=e^{2\4{\beta}T_0}\max (2,R_1,R_2(\lambda_1),R_2(\lambda_3), \lambda_3r_3').
\end{equation*}
Since  $g_{\lambda}$ is  continuous on $\R^n$  for any $\lambda>0$, 
 by the Kato inequality (\cite{DK}, \cite{K}), \eqref{def Ulam inv}, \eqref{uvU inv1}, \eqref{del rg n}, \eqref{eq-q-diff-rescaled-newn}, \eqref{eq-a-lower-bd-newn}, Corollary \ref{g-lambda-12-difference-cor} and Corollary \ref{rg-lambda'-g-lambda-ratio-cor}, for any $0<\3\le \3_1$, $R\ge R_3$ and $0<t<T_0$,
\begin{align}\label{u-bar-12-difference-t-derivative}
&\frac{d}{dt}\left(\int_{\re^n\setminus \{0\}}|\overline u_1-\overline u_2|(x,t)\eta_{\ve, R}(x) g_{\lambda_3}^{m \gamma_3}(x) \,dx\right)\notag\\
\le&\int_{\re^n\setminus \{0\}} |\overline u_1^m- \overline u_2^m|(x,t)\Delta\left( |x|^{n+2-\frac{n-2}{m}} \eta_{\ve, R}(x)g_{\lambda_3}^{m \gamma_3'}(x) \right) \, dx\notag\\
\le & \int_{\re^n\setminus \{0\}} |\overline u_1^m-\overline u_2^m|(x,t) \left(|x|^{n+2-\frac{n-2}{m}} g^{m \gamma_3}(x) \Delta \eta_{\ve, R}(x)+2\nabla \eta_{\ve, R} \cdot \nabla (|x|^{n+2-\frac{n-2}{m}} g_{\lambda_3}^{m \gamma_3}(x)) \right)\, dx.\notag\\
\le& C R^{n-\frac{n-2}{m}}\int_{B_{2R}\setminus B_R} \overline U_{\ld_1}^{m-1}(x,t) \left( \overline U_{\ld_2}(x,t)-\overline U_{\ld_1}(x,t) \right) g_{\lambda_3}^{m \gamma_3} (x)\, dx\notag\\
&\qquad + C \ve^{-2} \int_{B_{2\ve}\setminus B_\ve} \overline U_{\ld_2}^m(x,t)|x|^{n+2-\frac{n-2}{m}} g_{\lambda_3}^{m \gamma_3}(x) \,dx\notag \\
\le&C R^{n-\frac{n-2}{m}} e^{-m\4{\alpha}t}\int_{B_{2R}\setminus B_R} g_{\ld_1}^{m-1}(e^{-\4{\beta}t}x)\left( g_{\ld_2}(e^{-\4{\beta}t}x)-g_{\ld_1}(e^{-\4{\beta}t}x) \right)g_{\lambda_3}^{m \gamma_3} (x)\, dx\notag\\
&\qquad +C \3^{n-\frac{n-2}{m}}e^{-m\4{\alpha}t}\int_{B_{2\3}\setminus B_{\3}} g_{\ld_2}^m(e^{-\4{\beta}t}x)g_{\lambda_3}^{m \gamma_3}(x)  \, dx\notag\\
\le&C R^{n-\frac{n-2}{m}} e^{-m\4{\alpha}t}\int_{B_{2R}\setminus B_R}
\left(\frac{|e^{-\4{\beta}t}|x||^{\frac{n-2}{m}-n}}{\log (e^{-\4{\beta}t}|x|)} \right)
|e^{-\4{\beta}t}x|^{\frac{2}{1-m}-\frac{n-2}{m}}(\log (e^{-\4{\beta}t}|x|))^{\frac{m}{1-m}}(|x|^{-\frac{n-2-nm}{m}}\log |x|)^{\frac{m \gamma_3}{1-m}} \, dx\notag\\
&\qquad+C \3^{2n-\frac{n-2}{m}}e^{-m\4{\alpha}t}\notag\\
\le&CR^{n+\frac{2}{1-m}-\frac{n-2}{m}-\frac{n-2-nm}{1-m}\gamma_3}(\log R)^{\frac{m}{1-m}(1+\gamma_3)-1}+C \3^{2n-\frac{n-2}{m}}\notag\\
=&C(\log R)^{\frac{m}{1-m}(1+\gamma_3)-1}+C \3^{2n-\frac{n-2}{m}}
\end{align}
since $\4{\beta}T_0\le(1/2)\log R$. By integrating \eqref{u-bar-12-difference-t-derivative} over $(0,t)$, 

\begin{align}\label{u-bar-12-l1-difference}
&\int_{\re^n\setminus \{0\}}|\overline u_1-\overline u_2|(x,t)\eta_{\ve, R}(x) g_{\lambda_3}^{m \gamma_3}(x) \,dx\notag\\
\le &\int_{\re^n\setminus \{0\}}|\overline u_{0,1}-\overline u_{0,2}|(x)\eta_{\ve, R}(x) g_{\lambda_3}^{m \gamma_3}(x) \,dx+C(\log R)^{\frac{m}{1-m}(1+\gamma_3)-1}t+C \3^{2n-\frac{n-2}{m}}t
\end{align}
holds for any $0<\3\le \3_1$, $R\ge R_3$ and  $0<t<T_0$. Since by \eqref{m-n-relation2} and \eqref{gamma3-defn} we have
\begin{equation*}
\frac{m}{1-m}(1+\gamma_3)<1\quad\mbox{ and }\quad 2n>\frac{n-2}{m},
\end{equation*}
by letting first $\3\to 0$ and then $R\to\infty$ in \eqref{u-bar-12-l1-difference} we get that \eqref{inv uu0} holds for any $0<t<T_0$.
Since $T_0$ is artirary,  \eqref{inv uu0} holds and the proposition follows.
\end{proof}

\begin{cor}\label{u12-difference-weighted-l1-cor} 
Let $n$, $m$ satisfy \eqref{m-n-relation2} and  $\alpha$, $\beta$, $\widetilde \alpha$, $\widetilde \beta$, $\gamma_3$ be as given by \eqref{ab}, \eqref{tilde ab} and \eqref{gamma3-defn}, respectively.
Let $\lambda_1> \lambda_2>0$, $\lambda_3>0$, and $f_{\lambda_i}$, $U_{\lambda_i}, i=1,2,3$, be as given by \eqref{f-lambda-defn} and \eqref{def Ulam}, respectively, with $\lambda=\lambda_1,\lambda_2, \lambda_3$. Let $u_{0,1}$, $u_{0,2}$ be radially symmetric that satisfy \eqref{initial_condition_lower-upper-bd} and
\begin{equation}\label{u0-v0-weighted-l1 rf}
u_{0,1}-u_{0,2}\in L^1\left( |x|^{\frac{n-2}{m}+(n-2) \gamma_3-2n} f_{\lambda_3}^{m \gamma_3}(x); \re^n \setminus \{0\}\right).
\end{equation}
Let $u_1$, $u_2$ be the solutions of \eqref{main-eq} with initial values $u_{0,1}$, $u_{0,2}$, respectively which satisfy \eqref{uvU} and let $\4{u}_1$, $\4{u}_2$ be given by \eqref{rescaled-soln-defn} with $u=u_1,u_2$, respectively. Then
\begin{multline}\label{uu0a rf}
\int_{\R^n\setminus\{0\}}|u_1-u_2| (x,t)|x|^{\frac{n-2}{m}+(n-2) \gamma_3-2n}f_{\lambda_3}^{m \gamma_3}(x)\,dx\\
\le\int_{\R^n\setminus\{0\}}|u_{0,1}-u_{0,2}| (x)|x|^{\frac{n-2}{m}+(n-2) \gamma_3-2n }f_{\lambda_3}^{m \gamma_3}(x)\,dx \quad \forall t>0
\end{multline}
and
\begin{multline}\label{u-bar-difference-weight-l1}
\int_{\R^n\setminus\{0\}}|\4{u}_1-\4{u}_2| (x,t)|x|^{\frac{n-2}{m}+(n-2) \gamma_3-2n}f_{\lambda_3}^{m \gamma_3}(x)\,dx\\
\le\int_{\R^n\setminus\{0\}}|u_{0,1}-u_{0,2}| (x)|x|^{\frac{n-2}{m}+(n-2) \gamma_3-2n }f_{e^{-\beta t}\lambda_3}^{m \gamma_3}(x)\,dx \quad \forall t>0.
\end{multline}
\end{cor}
\begin{proof}
Since $u_{0,1}$, $u_{0,2}$ are radially symmetric, by uniqueness of solution of \eqref{main-eq} both $u_1$ and $u_2$ are radially symmetric in $\R^n\setminus\{0\}$ for any $t>0$. Let $\2{u}_1$, $\2{u}_2$ be given by \eqref{u-bar-defn} with $u=u_1, u_2$, respectively. Then $\2{u}_1$, $\2{u}_2$ are the solutions of \eqref{Inversed eq} with initial values $\2{u}_{0,1}:=|x|^{-\frac{n-2}{m}} u_{0,1}(|x|^{-1})$, $u_{0,2} :=|x|^{-\frac{n-2}{m}} u_{0,2}(|x|^{-1})$, respectively. By \eqref{u-bar-defn}, \eqref{U-lambda-U-bar-eqn}, and \eqref{uvU}, we have \eqref{uvU inv1}. Hence by Proposition \ref{u12-bar-g-lambda-weight-l1-prop}  we get that \eqref{inv uu0} holds.

Therefore, by \eqref{inv uu0}, we obtain \eqref{uu0a rf}, since
\begin{align*}
& \int_{\re^n \setminus \{0\} }|\overline u_{0,1}-\overline u_{0,2}|(x) g_{\lambda_3}^{m \gamma_3}(x) \, dx\notag\\
=&\omega_n\int_0^{\infty} r^{n-1-\frac{n-2}{m}-(n-2)\gamma_3}|u_{0,1}(r^{-1})- u_{0,2}(r^{-1})|f_{\lambda_3}^{m \gamma_3}(r^{-1})\,dr\notag\\
=&\omega_n\int_0^{\infty} \rho^{\frac{n-2}{m}+(n-2)\gamma_3-n-1}|u_{0,1}(\rho)- u_{0,2}(\rho)|f_{\lambda_3}^{m \gamma_3}(\rho)\,d\rho\notag\\
=&\int_{\re^n \setminus \{0\} }|u_{0,1}(y)- u_{0,2}(y)||y|^{\frac{n-2}{m}+(n-2)\gamma_3-2n}f_{\lambda_3}^{m \gamma_3}(y)\,dy
\end{align*}
and similarly,
\begin{equation*}
\int_{\re^n \setminus \{0\} }|\overline u_1-\overline u_2|(x,t) g_{\lambda_3}^{m \gamma_3}(x) \, dx=\int_{\re^n \setminus \{0\} }|u_1- u_2|(y,t)|y|^{\frac{n-2}{m}+(n-2)\gamma_3-2n}f_{\lambda_3}^{m \gamma_3}(y)\,dy\quad\forall t>0
\end{equation*}
holds.
Let $\delta_3= \frac{n-2}{m}+(n-2) \gamma_3 -2n$. Then by  \eqref{ab} and \eqref{gamma3-defn},
\begin{equation}\label{a-b-gamma3-relation}
\alpha - n\beta-\beta \delta_3+\frac{2\beta m\gamma_3}{1-m}=\beta\left(n-\frac{2}{1-m}\right)\left(\frac{n+\frac{2}{1-m}-\frac{n-2}{m}}{n-\frac{2}{1-m}}-\gamma_3\right)= 0. 
\end{equation} 
Hence by \eqref{f-lambda-defn}, \eqref{uu0a rf} and \eqref{a-b-gamma3-relation}, 
\begin{align*}
&\int_{\R^n\setminus\{0\}} |\widetilde u_1-\widetilde u_2|(x,t) |x|^{\delta_3} f_{\lambda_3}^{m \gamma_3}(x)\, dx\notag \\ 
=&e^{\alpha t}\int_{\R^n\setminus\{0\}}|u_1-u_2|(e^{\beta t} x,t)|x|^{\delta_3} f_{\lambda_3}^{m \gamma_3}(x)\, dx\notag\\
=&e^{(\alpha - n\beta)t}\int_{\R^n\setminus\{0\}}|u_1-u_2|(y,t) (e^{-\beta t}|y|)^{\delta_3} f_{\lambda_3}^{m \gamma_3}(e^{-\beta t}y)\, dy\notag\\
=&e^{\left(\alpha - n\beta-\beta \delta_3+\frac{2\beta m\gamma_3}{1-m}\right)t} \int_{\R^n\setminus\{0\}}|u_1-u_2|(y,t)|y|^{\delta_3}  \left(e^{-\frac{2\beta t}{1-m}} f_{\lambda_3}^{m \gamma_3}(e^{-\beta t}y)\right)\, dy\notag\\
=&\int_{\R^n\setminus\{0\}}|u_1-u_2|(y,t)|y|^{\delta_3}f_{e^{-\beta t}\lambda_3}^{m\gamma_3}(y)\,dy\notag\\
\le&\int_{\R^n\setminus\{0\}}|u_{0,1}-u_{0,2}| (x)|x|^{\frac{n-2}{m}+(n-2) \gamma_3-2n }f_{e^{-\beta t}\lambda_3}^{m \gamma_3}(x)\,dx \quad \forall t>0
\end{align*}
and \eqref{u-bar-difference-weight-l1} follows.
\end{proof}

By an argument similar to the proof of Theorem \ref{thm assm} but with Corollary \ref{u12-difference-weighted-l1-cor} replacing Corollary \ref{L1 con cor} in the proof there Theorem \ref{thm assm rad sym} follows.


\end{document}